\documentclass{amsart}

\usepackage{amssymb,latexsym}
\usepackage{amsmath,amscd}
\usepackage[all]{xy}
\usepackage{fancyhdr}
\usepackage{indentfirst}
\usepackage{graphicx}
\usepackage{newlfont}
\usepackage{amsfonts}
\usepackage{amsthm}
\usepackage{mathrsfs}
\usepackage{psfrag}
\usepackage[all]{xy}

\newtheorem{theorem}{Theorem}[section]

\theoremstyle{definition}

\theoremstyle{prop}
\newtheorem{prop}[theorem]{Proposition}

\theoremstyle{remark}
\newtheorem{remark}[theorem]{Remark}

\numberwithin{equation}{section}

\newcommand{\eps}       {\varepsilon}
\newcommand{\e}         {\varepsilon}
\renewcommand{\a}       {\alpha}

\renewcommand{\d}       {\partial}

\newcommand{\R}         {\mathbb{R}}

\newcommand{\C}         {\mathbb{C}}

\renewcommand{\P}       {\mathbb{P}}

\newcommand{\g}         {\mathfrak{g}}

\renewcommand{\epsilon} {\varepsilon}
\newcommand{\vol}{\textit{vol}}
\newcommand{\ann}{\textrm{Ann}}

\begin{document}

\title[Duistermaat--Heckman formula and cohomology of polygon spaces]{The
Duistermaat--Heckman formula and the cohomology of moduli spaces of polygons}


\author{Alessia Mandini}
\address{Center for Mathematical Analysis, Geometry and Dynamical Systems\\
Departamento de Matematica, Instituto Superior Tecnico\\
1049-001 Lisboa, Portugal\\
 Fax: (351) 21 8417035\\}
\email{amandini@math.ist.utl.pt}
\begin{thanks}
{Partially supported by the Funda\c{c}\~{a}o para a Ci\^{e}ncia e a Tecnologia
(FCT /
Portugal), grant SFRH / BPD / 44041 / 2008 and by Centro de An\'{a}lise
Matem\'{a}tica, Geometria e Sistemas Din\^{a}micos,
Departamento de Matem\'atica, IST, Lisbon (Portugal).}
\end{thanks}


\date{}

\dedicatory{}

\begin{abstract}

We give a presentation of the cohomology ring of spatial polygon spaces $M(r)$
with fixed side lengths $r \in \R^n_+$. 
These spaces can be described as the symplectic reduction of the Grassmaniann of
2-planes in $\C^n$ by the $U(1)^n$-action by multiplication, 
where $U(1)^n$ is the torus of diagonal matrices in the unitary group $U(n)$. We
prove that the first Chern classes of the $n$ line bundles 
associated with the fibration $ \textrm{($r$-level set)} \rightarrow M(r)$
generate the cohomology ring $H^* (M(r), \C).$ By applying the 
Duistermaat--Heckman Theorem, we then deduce the relations on these generators
from the piece-wise polynomial function that describes the 
volume of $M(r).$ 
We also give an explicit description of the birational map between $M(r) $ and
$M(r')$ when the lengths vectors $r$ and $r'$ are in different 
chambers of the moment polytope. This wall-crossing analysis is the key step to
prove that the Chern classes above are generators 
of $H^*(M(r))$ (this is well-known when $M(r)$ is toric, and by wall-crossing we
prove that it holds also when $M(r)$ is not toric).
\end{abstract}

\maketitle
\section{Introduction}

Spatial polygon spaces are a widely studied family of moduli spaces obtained by
symplectic reduction, see for example 
\cite{ag, Goldin, hk, hk1, kamiyama, km, khoi, klyachko, konno, io, ta01, ta02}.
A first way to introduce $M(r)$ is as 
the space of closed piece-wise linear paths in $\R^3$ such that the $j$-th step
has norm $r_j$, modulo rotations and translations. 
The vector $r=(r_1, \ldots, r_n) \in \R^n_+$ is called the lengths vector. 

Kapovich and Millson \cite{km} showed that one can describe $M(r)$ by means of a
symplectic quotient as follows. 
Let $\mathcal{S}_r = \prod_{j=1}^n S_{r_j}^2$ be the product of $n$ spheres in
$\R^3$ of radii $r_1, \ldots, r_n$ and centers all the origin. 
The diagonal $SO(3)$-action on $\mathcal{S}_r$ is Hamiltonian with moment map 
$\mu: \mathcal{S}_r \rightarrow \frak{so}(3)^* \simeq \R^3,$ $\mu(e_1, \ldots,
e_n)= e_1 + \cdots +e_n.$  
Note that an element $ ( e_1, \ldots, e_n ) \in \mathcal{S}_r$ is in
$\mu^{-1}(0)$ if and only if 
the path in $\R^3$ with edges 
$ e_1, \ldots, e_n$ closes to a polygon.
The moduli space of spatial polygons $M(r)$ arises then as the symplectic
quotient $
\mu^{-1}(0)/SO(3)=: \mathcal{S}_r/\!\!/_{\!0}SO(3).$ 

This fits into a broader picture: let $U(1)^n$ be the maximal torus of diagonal
matrices in the unitary group $U(n)$ and consider 
the action by conjugation of $U(1)^n \times U(2) \subset U(n) \times U(2)$ on
$\C^{n \times 2}$ (an element in the complex space 
is naturally thought as an $n\times 2 $ matrix). Note that the diagonal circle
$U(1) \subset U(1)^n \times U(2)$ fixes everything, 
and therefore only the action of $K := U(1)^n \times U(2) / U(1)$ is effective. 
One can then realize the polygon space $M(r)$ as the symplectic reduction 
$$\C^{n \times 2} \big/ \!\!\big/_{\!(r,0)} K$$
cf \cite{hk1}. It is enlightening to perform the symplectic reduction in stages.
Taking first the quotient by $U(1)^n$ one obtains 
the product of spheres $\mathcal{S}_r$ (here the reduction is performed by means
of the Hopf map as explained in Section \ref{inizio} 
and in \cite{hk1}). The residual $U(2)/U_1 \simeq SO(3)$ action is the one
described above, and one recovers the description of the 
polygon space $M(r)$ as the symplectic quotient $\mathcal{S}_r/\!\!/_{\!0}
SO(3).$ 

Performing the reduction in stages in the opposite order, one obtains the
Gelfand--MacPherson correspondence. In fact, one first obtains the
Grassmanian $Gr(2,n)$ of complex planes in $\C^n$ as 
the reduction $ \C^{n \times 2} /\!\!/_{\!0} U(2)$. Then the quotient by the
residual $U(1)^n / U(1)$ action on $Gr(2,n)$ is isomorphic to the moduli space
of $n$ points in $\C\P^1$, cf. \cite{gm}, and hence, by Klyachko \cite{klyachko}
and Kapovich and Millson \cite{km}, is also isomorphic to the 
polygon space $M(r)$. This is summarized in the following diagram:
$$ \xymatrix{
 {}& \C^{n \times 2} \ar[dl]_{U(2)} \ar[dr]^{U(1)^n}& {}\\
 Gr(2,n) \ar[dr]_{U(1)^n / U(1)}  & {} &  \prod_{j=1}^n S_{r_j}^2
\ar[dl]^{U(2)/U_1 \simeq SO(3)} \\
{}& M(r)   & {}
} $$

These two descriptions of the moduli space of polygons intertwine throughout the
paper, and give rise to the description we present of the 
cohomology ring of $M(r).$
On the subject there is a broad literature. Hausmann and Knutson \cite{hk}
computed the integer cohomology rings of the moduli spaces $M(r)$ 
by embedding these spaces (which in general are not toric) in toric varieties
and computing the kernel of the induced restriction map on 
cohomologies. 
The cohomology ring of the polygon space was also computed by Goldin.
In fact in \cite{Goldin} she finds explicit formulae for the rational cohomology
ring of the symplectic reduction of coadjoint orbits of $SU(n)$ by the action of
a maximal torus. Considering degenerate coadjoint orbits she determines the 
cohomology ring of the reduction of the Grassmannian of $k$-planes in $\C^n$.
By the Gelfand--MacPherson correspondence, see \cite{gm}, this is the moduli
space of $n$ points in $\C\P^{k-1}$ which, for $k=2$, is isomorphic to the 
moduli space of $n$-sided polygons in $\R^3$.
Previously Brion \cite{B} and Kirwan \cite{K} have computed the
rational cohomology ring of the special case of equilateral 
polygon spaces $M(1, \ldots, 1)$ with an odd number of edges. 

Many other contributed to the study of these
spaces, see for example \cite{ag, konno, ta01} where the 
intersection numbers are explicitly computed (by means of a recursion formula in
\cite{ag}, using ``quantization commutes with reduction'' 
in \cite{ta01}, via an algebro-geometric approach in \cite{konno}). More
contextualized reference will be given throughout the paper.

Our approach to the cohomology ring of $M(r)$ is as follows. First we give an
explicit description of the birational map between two 
polygon spaces $M(r)$ and $M(r')$ when $r$ and $r'$ lie in different chambers of
the moment polytope $\mu_{U(1)^n}^{-1}(Gr(2,n))$. 
Using this description we prove that the first Chern classes $c_1,\ldots, c_n $
of the $n$ line bundles associated to the 
fibration $\mu^{-1}_{U(1)^n}(r) \rightarrow M(r)$ generate the cohomology ring
$H^*(M(r), \C)$ (whenever $M(r)$ is a smooth manifold). 
This provides the opportunity to determine the relations on the generators $c_1,
\ldots, c_n$ by applying the Duistermaat--Heckman Theorem, 
once the piece-wise polynomial function $\textit{vol} \, M(r)$ that associates
to each generic $r$ the symplectic volume of $M(r)$ is 
known, as it is in our case. In particular this also proves that
polygon spaces only have even dimensional cohomology.
The fact that $H^*(M(r), \C)$ is generated by the classes $c_i$'s was
established in \cite[Corollary 7.4]{hk}, see also Remark \ref{hausknut}.

Let us describe in more detail the results in this paper. Section \ref{inizio}
is a brief overview on polygon spaces, where we give 
details for the symplectic reductions outlined above and for the moment polytope
$\Xi := \mu_{U(1)^n}^{-1}(Gr(2,n))$.
The polygon space $M(r)$ is a smooth K\"ahler manifold if and only if for 
any index set $I\subset \{1, \ldots, n\}$ the scalar quantity
$$\epsilon_I(r): = \sum_I r_i - \sum_{I^c} r_i$$
never vanishes. When this is the case, the lengths vector $r$ is called
\emph{generic}.

In Section \ref{vo} we prove that, for $r$ generic, the piece-wise polynomial
function for the symplectic volume of $M(r)$ is given by
\begin{equation}\label{v} 
\vol \, M(r)=- \, \frac{(2 \pi)^{n-3}}{2(n-3)!}\,  \sum_{I \, \textrm{long}}
(-1)^{n-|I|} \, \epsilon_I(r)^{n-3}
\end{equation}
where an index set $I$ is said to be long (or $r$-long) if and only if
$\epsilon_I(r) >0.$  
The symplectic volume of $M(r)$ was first computed by
Takakura  \cite{ta01} by means of a formula for the generating function of the
intersection pairings of $M(r)$. Formula \eqref{v} was later obtained
independently by Vu The Khoi in \cite{khoi}. The equilateral polygon space 
$M(1,\ldots,1)$ has some independent interest, and has been studied under
several points of view, see for example \cite{B, kirwan}. 
It is easy to see that the equilateral polygon space is smooth only for odd
number of edges $n$. In this case, the symplectic volume of $M(1,\ldots,1)$ has
been computed by Kamiyama and Tezuka \cite{kamiyama}, Takakura \cite{ta02} and
by Martin \cite{Martin}. 
In Section \ref{vo} we prove that Martin's techniques can be adapted to compute
the volume of $M(r)$ (Theorem \ref{volume}) for generic $r$'s.
We believe that this has some independent interest. The proof sets in the
context of equivariant cohomology, where the surjectivity of the Kirwan map $k:
H^*_{SO(3)}(\mathcal S_r) \rightarrow H^*(M(r)) $ suggests that the calculation
of the symplectic volume $\vol \, M(r)$ can be done by looking at $\int_{M(r)}
k(a)$ for a suitably chosen equivariant form $a \in H^*_{SO(3)}(\mathcal S_r)$
(i.e. such that $k$ maps $a$ onto the top power of the symplectic reduced form
$\omega_r$ on $M(r)$).
This is the natural setting for beautiful results, known as Localization
Theorems, that enable one to localize the computation of the integral above at
data associated to the fixed point set. Formula \eqref{v} is then an application
of Martin's Localization Theorem (cf. \cite{Martin} and Theorem
\ref{localization} in here).

In Section \ref{crossing the walls} we deal with describing the diffeotype of
$M(r)$ when $r$ crosses a wall in $\Xi$. It is well-known 
that for $r^0 $ and $r^1$ on either side of a wall, the associated symplectic
reductions $M(r^0)$ and $M(r^1)$ are related by a birational 
map which is the composite of a blow up followed by a blow down. This holds in
greater generality, as proven in \cite{gs89, brion}. For 
polygon spaces we can characterize the submanifolds blown up and blown down as
lower dimensional polygon spaces, cf. Theorem \ref{WCtheo}. 
The moment polytope $\Xi$, first studied in \cite{hk1}, is the hypersimplex 
$\{ r \in \R^n_+ \mid 0 \leq 2r_i \leq 1 \textnormal{ and } \sum_{i=1}^n r_i =
1\}$.
The regions $\Delta^{\! i}$ of regular values in $\Xi$ (called chambers)
are separated by walls,
which are the connected components of the image via 
$\mu_{U(1)^n}$ of the fixed points set $Gr(2,n)^{H}$ for subgroups $H \subset
U(1)^n$. It is not difficult to see that it is enough 
to consider the circles $\{ \textrm{diag} (e^{i \theta} \chi_I(1), \ldots, e^{i
\theta} \chi_I(n)) \} \subset U(1)^n $ where, for any 
index set $I\subset \{ 1, \ldots, n\}$ index set, $\chi_I(i)= 1 $ if $i \in I$
and $\chi_I(j)= 0$ if $j \in I^c$.
It follows (see Section \ref{inizio} and \cite{hk1}) that the walls in $\Xi$
have equation
\begin{equation}\label{eq:epsilon}
\sum_I r_i - \sum_{I^c} r_i=0
\end{equation}
for some $I\subset \{ 1, \ldots, n\}$. The polygon space $M(r)$ is a smooth
$(n-3)$-dimensional symplectic manifold if and only if equation 
\eqref{eq:epsilon} is never satisfied for any index set $I$. Consider 
$$ \epsilon_{I_p}(r)= \sum_{I_p} r_i - \sum_{I_q} r_i$$
where the index set $I_p = \{ i_1, \ldots i_p\}$ has cardinality $p$ and $ I_q =
\{ j_1, \ldots j_q\}$ is its complement (hence   $q:=n-p$).
Let $W_{I_p}$ denote the data of the wall of equation $\eps_{I_p}(r) = 0$
together with the wall-crossing direction from the chamber 
where $\eps_{I_p}(r) >0$  to the one where $ \eps_{I_q}(r)>0.$ Throughout this
paper we will only consider single wall-crossings, i.e. we 
assume that the wall-crossing point $r^c$ is not on a intersection of walls. 
This is not restrictive, since any non-single wall-crossing can 
be decomposed in a finite number of single wall-crossings. Let $r^c \in W_{I_p}$
be the wall-crossing point, i.e. $ \epsilon_{I_p}(r^c)=0$. 
It follows that $\mu^{-1}(0)$ contains the $SO(3)$-orbit of the polygon $P^c=
(e_1^c, \ldots, e_n^c)$ with 
\begin{displaymath}
e_i^c= \left\{
\begin{array}{c}
(r_i, 0,0) \textrm{ if } i \in I_p;\\
-(r_i, 0,0) \textrm{ if } i \in I_q.\\
\end{array}\right.
\end{displaymath}
The polygon $P^c$ lies completely on a line, the $x$-axis, and therefore it is
fixed by the circle $S^1$ of rotations around it. 
This originates a singularity of conic type in the quotient $M(r^c)$.
In Section \ref{crossing the walls} we analyze this singularity using the
description of $M(r)$ as the symplectic quotient of $Gr(2,n)$ 
by $U(1)^n$. 
In particular, we first perform reduction on $Gr(2,n) $ by a complement $H$ of
the circle $S^1 \subseteq U(1)^n / U(1)$ associated to the wall. 
The residual $S^1$-action on $Gr(2,n) /\!\!/_{ \!\!\!\{r_2, \ldots, r_{n-1} \}}
H$ is still Hamiltonian with moment map $\mu_{S^1}$.
The wall-crossing problem for polygon spaces gets then reduced to studying the
changes in the quotient
$$\big( Gr(2,n) /\!\!/_{ \!\!\{r_2, \ldots, r_{n-1} \}} H \big) /\!\!/_{\!\!r}
S^1$$
when $r$ goes through a critical value of $\mu_{S^1}$.
This provides us with two blow down maps $\beta_- : M(r^0) \rightarrow M(r^c)$
and $\beta_+ : M(r^1) \rightarrow M(r^c)$, where $r^0$ 
and $r^1$ are regular values respectively before and after the wall-crossing as
above. To give an explicit description of the two maps 
$\beta_-$ and $\beta_+$ let us introduce some notation. Consider 
the submanifold $M_{I_p}(r) \subset M(r)$ of polygons such that the edges $e_i$,
for $i \in I_p$, are parallel and point in the same 
direction as in Figure \ref{Mp}.

\begin{figure}[htbp]
\begin{center}
\psfrag{1}{\footnotesize{$e_1$}}
\psfrag{2}{\footnotesize{$e_2$}}
\psfrag{3}{\footnotesize{$e_3$}}
\psfrag{4}{\footnotesize{$e_4$}}
\psfrag{5}{\footnotesize{$e_5$}}
\psfrag{6}{\footnotesize{$e_6$}}
\psfrag{7}{\footnotesize{$e_7$}}
\psfrag{8}{\footnotesize{$e_8$}}
\psfrag{9}{\footnotesize{$e_9$}}
\includegraphics[width=3cm]{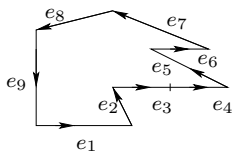}
\end{center}
\caption{\footnotesize{A polygon in $M_{I_4}(r)$ with $I_4= \{ 1,3,4,6\} \subset
\{ 1, \ldots, 9\}$ }}
\label{Mp}
\end{figure}
Note that $M_{I_p}(r)$ is naturally isomorphic to the moduli space $M(r_{I_p})$
of $(p+1)$-gons with lengths vector $r_{I_p}:=(\sum_{i \in I_p} r_i, r_{j_1},
\ldots, r_{j_q})$. It follows that $M_{I_p}(r^0)$ is empty. In fact the
condition $\epsilon_{I_p} (r^0) >0$ implies that $\{1\}$ is an 
$r_{I_p}$-long edge, and therefore the closing condition $\sum e_i =0$ is never
satisfied for
any $\vec{e} \in \mathcal S_{r_{I_p}}$. On the other hand, $M_{I_q}(r^0) \simeq
M(r_{I_q})$ is not empty and, as proven in Section \ref{crossing the walls}, is
diffeomorphic to the projective space $\C\P^{p-2}$.
Similarly, the submanifold $M_{I_q}(r^1)$ is empty while  $M_{I_p}(r^1)$ is the
projective space $\C\P^{q-2}$.
In Section \ref{crossing the walls} we prove that $\beta_-$ maps $ M(r^0)
\setminus M_{I_q}(r^0)$ diffeomorphically onto $ M(r^c) \setminus [P^c]$, and $
M_{I_q}(r^0)$ gets blown down via $\beta_-$ to $[P^c]$.
Similarly, $\beta_+$ blows down $M_{I_p}(r^1)$ to $[P^c]$ and maps $ M(r^1)
\setminus M_{I_p}(r^1)$ diffeomorphically onto $ M(r^c) \setminus [P^c]$. 
The spaces $M_{I_p}(r^1)$ and $M_{I_q}(r^0)$ are different resolutions of the
singularity corresponding to the degenerate polygon $[P^c]$ in $M(r^c),$ and
both are dominated by the blow up $\widetilde{M}$ of $M(r^c)$ at the singular
point, with exceptional divisor $E \simeq \C\P^{q-2} \times \C\P^{p-2} \simeq
M_{I_p}(r^1) \times M_{I_q}(r^0)$.

This wall-crossing analysis and the volume formula intertwine in Section
\ref{coho} where we describe the cohomology ring $H^*(M(r), \C).$ In Section
\ref{The cohomology ring of reduced spaces} we recall some results due to
Guillemin and Sternberg \cite{gs95} on the cohomology ring of reduced spaces.
Section \ref{wallcrossing and cohomology} is the heart of our description of
$H^*(M(r)).$ In fact, if $r^0$ is in an external chamber of the moment polytope
$\Xi,$ then $M(r^0)$ is toric and the Chern classes $c_1, \ldots, c_n$ of the
$n$ complex line bundles associated with $\mu_{U(1)^n}^{-1} (r^0) \rightarrow
M(r^0)$ generate the cohomology ring $H^*(M(r^0)).$ We prove by wall-crossing
arguments that this holds for any regular value $r$ in any chamber of the moment
polytope $\Xi.$ In fact any internal chamber can be reached from an external one
by a finite number of single wall-crossings. First we prove, as an application
of our wall-crossing analysis, that crossing a wall $W_{I_p}$ changes the
dimensions of the cohomology groups of degree $k$ for $k$ an even integer in the
interval  $[2 \textrm{min}(p,q) -2 , 2 \textrm{max}(p,q) -4]$. Precisely, if
$M(r^0)$ and $M(r^1)$ are polygon spaces before and after crossing the wall 
$W_{I_p},$ then
$$\textnormal{dim} H^k(M(r^1))= \textnormal{dim}H^k(M(r^0)) +1, \quad  2p-2 \leq
k \leq 2q-4 \,(\textnormal{case} \, q \geq p )$$
for $k$ even, and $H^k(M(r^1)) = H^k(M(r^0))$ for any other value of $k.$ 
In particular
$$H^k(M(r^1)) = H^k(M(r^0))= 0 \quad \forall k \textnormal{ odd }.$$
This result may also be obtained from the
Poincar\'e polynomial formulae of Klyachko \cite{klyachko} or Hausmann-Knutson
\cite{hk}.
This increasing in the dimension of the ``middle'' cohomology groups can be
explicitly described in terms of the submanifolds $M_{I_p}(r^1)$ and 
$M_{I_q}(r^0).$ Assume for simplicity that $q \geq p,$ and let $PD
([M_{I_p}(r^{0})])$ be the Poincar\'e dual of $M_{I_p}(r^0).$ Then the
cohomology class $PD([M_{I_p}(r^0)])$ living in $H^{2p-2}(M(r^0))$ is zero,
since $M_{I_p}(r^0)$ is empty.
On the other side $M_{I_p}(r^1)$ is not empty and the class of its Poincar\'e
dual $PD([M_{I_p}(r^1)])$ determines the increase in the dimension of
$H^{2p-2}(M(r^1)).$ 

The increase in higher dimensional cohomology groups is given by the cup
product $PD([M_{I_p}(r^{1})]) \smile c_1^{\alpha}(\mathcal N^{1})$ of
$PD([M_{I_p}(r^1)])$ with the first Chern class $c_1(\mathcal N^{1})$, where
$\mathcal N^1$ is the normal bundle to $M_{I_p}(r^1)$, and with its powers
$c_1^{\alpha}(\mathcal N^{1})$ for $\alpha=0, \ldots, q-p,$ as prescribed by the
Decomposition Theorem (see \cite{bbd} and also \cite{luca}, and Theorem
\ref{decomposition} in here):
$$H^*(M(r^1)) = H^*(M(r^0))\oplus \bigoplus_{\a=0}^{q-p} \C \Big(
PD([M_{I_p}(r^1)]) \smile c_1^{\a} (\mathcal N^1) \Big).$$
We prove that $PD([M_{I_p}(r^1)])$ and $c_1(\mathcal N^1)$ are linear
combinations of the
first Chern classes $c_1, \ldots, c_n,$ (cf. Proposition \ref{MIP}). It follows
that $H^*(M(r^1))$ is generated by $c_1, \ldots, c_n$ if $H^*(M(r^0))$ is as
well. This is the case for $r^0$ in an external chamber and therefore, by
crossing a finite number of walls, for $r^0$ in any chamber of the moment
polytope $\Xi.$ 

Applying the Duistermaat--Heckman Theorem one can then describe the cohomology
ring of $M(r)$ as follows 
(Theorem \ref{fine}): $$ H^*(M(r), \C) \simeq \C[x_1, \ldots, x_n]/
\textrm{Ann}(\vol \, M(r))$$ where a polynomial 
$Q(x_1, \ldots, x_n) \in \textrm{Ann}(\vol \, M(r)) $  if and only if $$Q \Big(
\frac{\d}{\d r_{1}}, \ldots, \frac{\d}{\d r_{n}} \Big) \vol \, M(r)=0.$$

{\bf Acknowledgements.} I am grateful to Luca Migliorini for the insightful
discussions and support. I also wish to thank ICTP in Trieste and CRM ``Ennio de
Giorgi'' in Pisa for providing excellent working conditions that made these
discussions possible. I also wish to thank Leonor Godinho for helpful
suggestions, and the University of Utrecht for its hospitality. 
Finally I would like to thank the referees for several improving comments.

\section{The Moduli Space of Polygons $M(r)$ }\label{inizio}

Let $S^2_t$ be the sphere in $\R^3$ of radius $t$ and center the origin.  For
$r= (r_1, \ldots, r_n) \in \R^n_+,$ the product $\mathcal{S}_r = \prod_{j=1}^n
S_{r_j}^2$ of $n$ 2-spheres is a smooth manifold. Let $p_j\colon \mathcal{S}_r
\rightarrow S_{r_j}^2 $ be the projection on the $j$-th factor and let $\omega_j
$ be the volume form on the sphere $S_{r_j}^2.$ Because the $\omega_j$'s are
closed and non-degenerate,  the $2$-form $\omega= \sum_{j=1}^n \frac{1}{r_j}
p_j^* \omega_j$ is closed and non-degenerate as well and defines a symplectic
structure on  $\mathcal{S}_r.$ 
Note that the symplectic form $\omega$ can be written equivalently as 
$\sum_{j=1}^n r_j \pi_j^* \omega_{S^2}$ where $\pi_j$ is the composition 
of $p_j$ with the rescaling map $S_{r_j}^2 \to S^2$, for details see
\cite[Section 1]{km}.

The group $SO(3)$ acts diagonally on $\mathcal{S}_r.$ Equivalently,
identifying the sphere $S_{r_j}^2$ with a $SO(3)$-coadjoint orbit,  the
$SO(3)$-action on each sphere is the coadjoint one. The choice of an invariant
inner product on the Lie algebra $\mathfrak{so}(3)$ of $SO(3)$ induces an
identification $\mathfrak{so}(3)^* \simeq \R^3$ between the dual of
$\mathfrak{so}(3)$ and $\R^3.$ On each sphere $S^2_{r_j},$ the moment map
associated to the coadjoint action is the inclusion of $S^2_{r_j}$ in $\R^3.$ By
linearity, the diagonal action of $SO(3)$ on $\mathcal{S}_r$ has moment map 
\begin{displaymath}
\begin{array}{rcl}
\mu : \mathcal{S}_r& \rightarrow &\R^3\\
\vec{e}= (e_1, \ldots, e_n) & \mapsto & e_1 + \cdots + e_n.\\
\end{array}
\end{displaymath}
The level set $\mu^{-1}(0):= \widetilde{M}(r)= \{ \vec{e}=(e_1, \ldots, e_n) \in
\mathcal{S}_r \mid \sum_{i=1}^n e_i =0 \}$ is a submanifold of $\mathcal{S}_r$
because $0$ is a regular value for $\mu.$

A polygon in $\R^3$ is a closed piece-wise linear path in $\R^3$. Consider the
piece-wise linear path such that the $j$-th step is given by the vector $e_j$.
Such a path closes if and only if $\sum_{i=1}^n
e_i =0$. Therefore $\widetilde{M}(r)$ is the space of $n$-gons of fixed sides
length  $r_1, \ldots, r_n$. Its quotient 
$$M(r):= \widetilde{M}(r)/SO(3) = \mathcal{S}_r
 / \!\! / SO(3)$$ is the space of $n$-gons of fixed sides length  $r_1,
\ldots, r_n$ modulo rigid motions, and is usually called polygon space.
A polygon is called degenerate if it lies completely on a line.

The moduli space $M(r)$ is a smooth manifold if and only if the lengths vector 
$r$ is generic, i.e. for each $I \subset \{ 1, \ldots, n \},$ the quantity 
$$ \epsilon_I(r):= \sum_{i \in I} r_i - \sum_{i \in I^c} r_i$$ is non-zero.
Equivalently, if and only if in $M(r)$ there are no degenerate polygons. In
fact, if there exists a polygon $P$ on a line (or an index set $I$ such that
$\epsilon_I(r)=0$) then its stabilizer is $S^1$ since the polygon $P$ is fixed
by rotations around the axis it defines. Therefore the $SO(3)$-action on
$\widetilde{M}(r)$ is not free and the quotient has singularities, which have
been studied by Kapovich and Millson in \cite{km}. Precisely, they proved that
$M(r)$ is a complex analytic space with isolated singularities corresponding to
the degenerate $n$-gons in $M(r),$ and these singularities are equivalent to
homogeneous quadratic cones. Along the proof of the wall-crossing Theorem
\ref{WCtheo} in Section \ref{crossing the walls} we will provide an explicit
description of the cone $C_W$ over the singularity.
Note that, for $r$ generic, the polygon space $M(r)$ inherits a symplectic form
by
symplectic reduction, see for example \cite{audin}.

An alternative description of the moduli space $M(r)$ is given by
Hausmann and Knutson in \cite{hk1} which also resemble an earlier work of
Gelfand and
MacPherson \cite{gm}. 
With minor adaptations we provide an overview here.
Let $U(1)^n $ be the maximal torus of diagonal matrices in the unitary group
$U(n)$. The group $U(1)^n \times U(2)$ acts by conjugation on $ \C^{n \times
2}$. The action is Hamiltonian and the polygon space $M(r)$ can then be realized
as the symplectic quotient of $ \C^{n \times 2}$ by $U(1)^n \times U(2)$, cf.
\cite{hk1}. One can perform reduction in stages. Consider first
the $U(2)$-action 
with associated moment map $$\begin{array}{rcl}
\mu_{_{U(2)}} : \C^{n \times 2}& \rightarrow & \mathfrak{u}(2)^*\\
A& \mapsto & - \frac{i}{2}  (A^*A -Id),\\
\end{array}$$ where $A^*$ is the conjugate transpose of $A$ and $Id$ is the
identity matrix.
The Stiefel manifold of orthonormal 2 frames in $\C^n,$ defined as follows
$$ St_{2,n} = \Bigg\{ \left(\begin{array}{cc} 
a_1 & b_1\\
\vdots & \vdots \\
a_n& b_n
\end{array} \right) \in \C^{n \times 2}:  \sum_{i=1}^n |a_i|^2 =1,\, 
\sum_{i=1}^n |b_i|^2 =1, \, \sum_{i=1}^n a_i\bar{b_i}=0 \Bigg\} $$ can be
realized as the zero level set $\mu_{_{U(2)}}^{-1} (0)$.
Let $Gr(2,n)$ be the Grassmannian of 2-planes in $\C^n.$ The map $$ p:
St_{2,n}\rightarrow Gr(2,n)$$ that takes an element $(a,b) \in St_{2,n}$ into
the plane generated by the column vectors $a$ and $b$ is actually the projection
of $St_{2,n}$ on to the orbit space $St_{2,n}/U(2).$ This realizes the
Grassmannian
$Gr(2,n)$ as the symplectic quotient $\C^{n \times 2}/ \! \!/U(2).$ 
The projection $p$ is $U(n)$-equivariant and thus the $U(n)$-action descends to
an action on the quotient  $Gr(2,n)$.  
The action of the maximal torus $U(1)^n$ on $Gr(2,n)$ is Hamiltonian with
associated moment map $\mu_{_{U(1)^n}}: Gr(2,n) \rightarrow \R^n$ such that, if
$\Pi= \langle a,b \rangle$ is the plane generated by $a,b \in \C^n,$ then $$
\mu_{_{U(1)^n}}( \Pi)= \frac{1}{2}(|a_1|^2 + |b_1|^2, \ldots , |a_n|^2 +
|b_n|^2).$$
The image of the moment map $\mu_{_{U(1)^n}} (Gr(2,n))$ is the hypersimplex
$\Xi$
$$\mu_{_{U(1)^n}} (Gr(2,n))= \Xi = \Big\{ (r_1, \ldots, r_n) \in \R^n | \, 0
\leq 2 r_i \leq 1, \quad \sum_{i=1}^n r_i =1 \Big\} $$
and the set of critical values of $\mu_{_{U(1)^n}}$ consists of those points $
(r_1, \ldots, r_n) \in \Xi$ satisfying one of the following conditions
\begin{itemize}
\item[(a)] one of the $r_i$'s vanishes or  is equal to 1/2;
\item[(b)] $\exists$ $I$ such that $\epsilon_I(r)=0$ with $|I|$ and $|I^c|$ at
least two.
\end{itemize}

\begin{remark}\label{inner walls}
Note that points satisfying (a)  constitute the boundary of $\Xi,$ while points
satisfying condition (b) are the inner walls of $\Xi.$ 
Moreover condition (a) is equivalent to the following 
\begin{itemize}
 \item[(a')] $\exists$ $I$ such that $\epsilon_I(r)=0$ with $|I|=1$ or $|I|=n-1$
\end{itemize}
Therefore a wall in $\Xi$ has equation
\begin{equation}\label{eq:wall}
\epsilon_I(r) = - \epsilon_{I^c}(r)=0
\end{equation}
for some index subset $I \subset \{ 1, \ldots, n\}$. Denote by $W_I$ or $
W_{I^c}$ the wall of equation \eqref{eq:wall}. (In Section \ref{crossing the
walls} the choice of either $W_I$ or $W_{I^c}$ will encode the wall crossing
direction).
The walls separates the regions $\Delta^{\! i}$ of regular values, 
called chambers, for which
$\varepsilon_I (r) \neq 0$ for any $I \subset \{ 1, \ldots, n\}$. 
An index set $I$ is said to be \emph{short} if  $\epsilon_I(r)<0$,
and \emph{long} if its complement is short.
Geometrically, an index set $I$ is short if the polygon space $M(r)$ 
contains configurations $[e_1, \ldots, e_n]$ where the edges $e_i$, for 
$i \in I$, are all positive proportional to each other.
For example, a polygon as in Figure \ref{Mp} exists in $M(r)$ if and only 
if the index set $I_4= \{ 1,3,4,6\}$ is short.

Since for any regular value either $I$ or $I^c$ is short, and this is consistent
within the chamber $\Delta^{\! i}$ containing $r$, it follows that a chamber
$\Delta^{i}$ is uniquely determined by the collection of short sets 
$$\mathcal S( \Delta^{\! i}):= \{ I \subset \{ 1, \ldots, n\} \mid I
\textnormal{ is short for any } r \in \Delta^{\! i}\}.$$
Given $I \in \mathcal S( \Delta^{\! i})$, the wall $W_I$ is in the closure of
$\Delta^{\! i}$ if and only if $ I$ is maximal (with respect to the inclusion)
in $\mathcal S( \Delta^{\! i})$. A chamber $\Delta^{\! i}$ is external if its
closure contains an outer wall. Equivalently, this means that there exists a
cardinality-$1$ set $\{ j\} \in \mathcal S( \Delta^{\! i})$ which is not
contained in any other short set.
\end{remark}

Under a canonical diffeomorphism between $M(r)$ and $M(\lambda r)$ the
symplectic forms are proportional ($\omega_{\lambda r} = \lambda \omega_r$).
Hence the condition that fixes the perimeter $ \sum_{i=1}^n r_i =1$ is not
restrictive
and allows one to work with the compact polytope $\Xi$ rather than with the
positive octant $\R^n_+$, whose chambers of regular values are cones. The
particular choice $ \sum_{i=1}^n r_i =1$ descends from the moment map
$\mu_{_{U(2)}}$ (or equivalently from considering orthonormal frames in
$St_{2,n}$).

\begin{prop}(Hausmann--Knutson \cite{hk1})
For generic $r \in \Xi$, the polygon space $M(r)$ is the symplectic reduction
relative to the
$U(1)^n$-action on the Grassmaniann $Gr(2,n)$ at the level set $r,$ i.e.  $$
M(r) \simeq U(1)^n\backslash \mu_{_{U(1)^n}}^{-1}(r) =  Gr(2,n)/ \!\! /_{r}
U(1)^n .$$ 
\end{prop}

Note that one recovers the previous description of $M(r)$ as the symplectic
quotient  $\mathcal{S}_r  / \!\! / SO(3)$ performing the reduction of $ \C^{n
\times 2}$ by $U(1)^n \times U(2)$ in the opposite order.
In fact, let $\tilde{\mu}_{U(1)^n}: \C^{n \times 2} \rightarrow \R^n $ be the
moment map for the $U(1)^n$ action on $ \C^{n \times 2}$. Clearly 
$$ \tilde{\mu}_{U(1)^n}(a,b) = \frac{1}{2} (|a_1|^2 + |b_1|^2, \ldots , |a_n|^2
+ |b_n|^2)$$
and 
$$ \tilde{\mu}_{U(1)^n}^{-1}(r) \simeq  \prod_{j=1}^n S^3_{\sqrt{2 r_j}}.$$
The torus $U(1)^n $ acts diagonally on $\prod_{j=1}^n S^3_{\sqrt{2 r_j}}$ and
the projection map $$\tilde{\mu}_{U(1)^n}^{-1}(r) \rightarrow \C^{n \times 2} /
\!\! /_{\! r} \, U(1)^n$$ 
is just the Hopf map 
\begin{displaymath}
\begin{array}{rcl}
H^n: \prod_j S^3_{\sqrt{2 r_j}} & \rightarrow&  \prod_j S^2_{r_j}\\
(a,b)& \mapsto & (H(a_1,b_1), \ldots, H(a_n,b_n))
\end{array}
\end{displaymath}
where
$$ H(a_{i}, b_{i})= \Big(\frac{|a_i|^2 - |b_i|^2}{2}, \textnormal{Re} (\bar{a}_i
b_i), \textnormal{Im} (\bar{a}_i b_i) \Big)$$ 
and $ Re(\bar{a}_i b_i)$ and $Im(\bar{a}_i b_i)$ are the real and imaginary part
of $\bar{a}_i b_i \in \C.$
Note that on each sphere the map $H^n$ is obtained from the quaternionic Hopf
map 
$$ \tilde{H}(a_{\ell}, b_{\ell})= i[(|a_{\ell}|^2- |b_{\ell}|^2) +
2\bar{a}_{\ell} b_{\ell} \,j].$$
The residual $U(2) / U(1) \simeq SO(3)$ action is then the one described above.

Since these two descriptions of $M(r)$ obtained by performing reduction in
stages in different order will play a central role along the paper, we find it
convenient to explore here the relation between the two.
Denote by $p^{-1}(\mu_{U(1)^n}^{-1}(r))$ the preimage in $ St_{2,n}$ of the
$r$-level set in $Gr(2,n).$ Then $p^{-1}(\mu_{U(1)^n}^{-1}(r))$ is the set of
$(a,b) \in St_{2,n}$ such that each row has norm $2 r_i,$ i.e. 
$$p^{-1}(\mu_{U(1)^n}^{-1}(r)) = \{ (a,b) \in  St_{2,n}: |a_i|^2 +|b_i|^2 = 2
r_i \quad \forall i=1, \ldots n \}. $$
This naturally defines the inclusion map 
\begin{equation}\label{3sphere}
\imath: p^{-1}(\mu_{U(1)^n}^{-1}(r)) \hookrightarrow
\tilde{\mu}_{U(1)^n}^{-1}(r) \simeq  \prod_j S^3_{\sqrt{2 r_j}}.
\end{equation}
It follows that 
$$ H^n (\imath(p^{-1}(\mu_{U(1)^n}^{-1}(r))) )=  \mu_{SO(3)}^{-1}(0)= \{ (e_1,
\ldots, e_n) \in  \prod_j S^2_{r_j} \mid \sum_{i} e_i=0\}.$$
The fact that the vectors $e_i := H(a_i, b_i)$ sum up to 0 follows from the
conditions for $(a,b) \in St_{2,n}$:
$$H(a_i,b_i)=\Big(\sum_{i=1}^n \frac{|a_i|^2 - |b_i|^2}{2},  \sum_{i=1}^n
Re(\bar{a}_i b_i),  \sum_{i=1}^n Im(\bar{a}_i b_i) \Big)= 0.$$ 
Thus the following diagram embodies the rich geometric
structure of $M(r):$ 

\begin{equation}\label{beauty}
\xymatrix{St_{2,n} \supseteq p^{-1}(\mu_{U(1)^n}^{-1}(r)) \ar[d]_p
\ar@{^{(}->}[rr]& & \imath (p^{-1}(\mu_{U(1)^n}^{-1}(r))) \subseteq \prod_j
S^3_{\sqrt{2 r_j}} \ar[d]_{H^n} \\
Gr(2,n)\supseteq \mu_{U(1)^n}^{-1}(r) \ar[dr]_{U(1)^n} & &  \mu_{SO(3)}^{-1}(0)
\subseteq  \prod_j S^2_{r_j} \ar[dl]_{s}^{\,\,\,\,\,\, U(2)/U_1 \simeq SO(3)}\\
 & M(r) & \\}
\end{equation}

\begin{remark}\label{bending}
In \cite{km} Kapovich and Millson define the so called bending flows on the
polygon space $M(r)$. As proved in \cite{hk1}, these form the residual 
torus action from the Gelfand-Cetlin system on the Grassmannian  $Gr(2,n)$.
The bending flows define a toric action on the polygon space $M(r)$
if and only if it is possible to choose a system of $n-3$ non-intersecting and
nowhere vanishing
diagonals. For $n=4,5,6$, Hausmann and Knutson \cite[Section 6]{hk1}  determine
explicit combinatorial conditions depending on $r \in \R^n_+$
for the bending action on $M(r)$ to be toric, see also \cite{io}.
In particular, for $n=5$, in \cite{hk2},
Hausmann and Knutson prove that $M(1,1,1,1,1)$ is not toric. In 
fact it has Riemann-Roch number $6$ and Euler characteristic 7.
Still, for small $\epsilon$, the length vector $(1+\epsilon,1,1,1,1+\epsilon)$ 
is in the same chamber as $(1,1,1,1,1)$ and the bending flows define a toric 
action on the polygon space $M(1+\epsilon,1,1,1,1+\epsilon)$, cf. \cite{hk1,
hk2}.
Consequently in the same chamber we can obtain both toric and non-toric
manifolds, i.e. being toric is not an invariant of the chamber.
\end{remark}

\section{The Symplectic Volume of $M(r)$}\label{vo}
The goal of this of this section is to prove an explicit formula
for the volume of polygon spaces. 
The volume of $M(r)$ had been already computed \cite{ta01, khoi}, still we
believe that the proof we give via Martin's localization Theorem has some
independent interest. Moreover, the volume formula Theorem \ref{volume} has a
central role for our
description of the cohomology ring $H^*(M(r), \C)$ in Theorem \ref{fine}.
\subsection{Martin's Results}

In this section we give some basic definitions and results in equivariant
cohomology. On this topic there is a rich literature, in particular we refer to
the survey papers \cite{AB} and \cite{D}, and also the book \cite{K}.

Let $G$ be a compact Lie group acting on a smooth manifold $M$ in a Hamiltonian
way, with moment map $\mu:M \rightarrow \g^*.$ The equivariant cohomo\-logy of
$M$ is defined to be the ordinary cohomology of $M_G:= EG \times_G M,$ where
$EG$ is the total space of the universal bundle $EG \rightarrow BG,$ $BG $ being
the classifying space of the group $G.$ 

Let $\xi \in (\mathfrak g^*)^G$ be a regular value for the moment map $\mu,$
fixed by the co-adjoint $G$-action. Assume also that $G$ acts freely on
$\mu^{-1}(\xi),$ so that the orbit space $\mu^{-1}(\xi)/G := M/ \!\! /_{_{\xi}}
G $ is a manifold.

In \cite{K}, Kirwan proved that there exists an epimorphism
$$ k: H^*_G(M) \rightarrow H^*(M/\!\!/_{_{\xi}}G),$$ which is known as the
\emph{Kirwan map}.

The surjectivity of the Kirwan map rises the hope that a good deal of
information about the cohomology ring $H^*(M/\!\!/_{_{\xi}}G)$ of a reduced
space can be computed from the equivariant cohomology $H^*_G(M)$ of $M.$ The
extra information encoded by the equivariant cohomology turns out to be related
to the orbit structure of the $G$-action, and in this sense equivariant
cohomology is the natural setting for results, known as localization Theorems,
which enables many computation to be reduced to the fixed point set of the
$G$-action. 

Our proof of the volume formula (Theorem \ref{volume}) for the moduli space of
polygons is based on a localization Theorem due to Martin \cite{Martin}. A
similar result has been proven independently by Guillemin and Kalkman \cite{gk}.
In \cite{Martin} it is calculated, as an example, the symplectic volume of the
moduli space $M(1, \ldots,1)$ of polygons with an odd number of edges all of
length 1. In Section \ref{volu} we prove that, mutata mutandis, Martin's
techniques hold for any generic $r \in \R^n_+.$

Assume now that $M$ is symplectic. Endow $M$ with a Hamiltonian action of a
torus $T$ with associated moment map $\mu\colon X \rightarrow \mathfrak{t}^*$. 
Let $p_0$ and $p_1$ be two regular values of the moment map $\mu.$ A
\emph{transverse path $Z$} is a one-dimensional submanifold $Z \subset
\mathfrak{t}^*$ with boun\-dary $\{ p_0,p_1\}$ such that $Z$ is transverse to
$\mu.$
A \emph{wall} in $\mathfrak{t}^*$ is defined to be a connected component of
$\mu(M^H)$ where $M^H$ is the fixed point set for some oriented subgroup $H
\simeq S^1$ of $T.$

Orient $H$ as follows: first orient $Z$ from $p_0$ to $p_1.$ Each positive
tangent vector field in $T_qZ,$ thought as an element of $\mathfrak{t}^*,$
defines a functional on $\mathfrak{t}$ which restricts to a nonzero functional
on $\mathfrak{h}$:= Lie($H$). The orientation of $H$ is defined to be the
positive one with respect to this functional.

\begin{theorem}\label{localization}(\emph{Localization Theorem} \cite{Martin})
Let $p_0 $ and $p_1$ be regular values of the moment map $\mu$ joined by a
transverse path $Z$ having a single wall crossing at $q$ and let $H \simeq S^1$
be the oriented subgroup associated to the wall-crossing from $p_0$ to $p_1$.
There exists a map $$ \lambda_H : H^*_T (M) \rightarrow H^*_{T/H} (M^H),$$
called localization map, such that, for any $a \in H^*_T (X), $
$$ \int_{M/\!\!/_{_{\!p_0}}T } k_0(a) -\int_{M/\!\!/_{_{\!p_1}}T } k_1(a)  =
\int_{M^H /\!\!/_{_{\!q}}T } k_q(\lambda_H(a_{|_{M^H}}))$$
where the maps $k_i: H^*_T (M) \rightarrow H^*(M/ \!\! /_{_{\!p_i}}T) $ are the
Kirwan maps, $M^H /\!\!/_{_{\!q}}T $  is the symplectic quotient of
$\mu_{|_{M^H}}^{-1}(q) \cap M^H$ by the quotient subgroup $T/H$ and $k_q:
H^*_{T/H} (M^H) \rightarrow H^*(M^H/ \!\! /_{_{\!q}}T)$ is the associated Kirwan
map.
\end{theorem}

It is possible to describe the localization map $\lambda_{H} $ in terms of
equivariant characteristic classes. 
As pointed out in \cite{AB}, the functorial nature of the construction that to
$M$ associates $M_G$ enables one to define equivariant correspondents of the
concepts of ordinary cohomology in a natural way. In particular, if $V$ is a
vector bundle over $M,$ then any action of $G$ on $V$ lifting the action on $M$
can be used to define a vector bundle $V_G= EG \times_G V$ over $M_G$ that
extends the bundle $V \rightarrow M.$ Thus, for example, the first Chern class
of $V_G,$ $c_1(V_G),$ naturally lies in $H^*(M_G)=:H^*_G(M)$ and is called the 
equivariant first Chern class of $V,$ denoted by $c_1^{G}(V).$ All other
equivariant characteristic classes are defined in a similar way. (See also
\cite{Martin}, Appendix B.)

As before, $H\simeq S^1$ is the subgroup of $T$ associated to the wall-crossing
we are examining. Let $T' \subset T$ be a complement of $H,$ i.e. $T=T'\times
H.$ This defines an isomorphism $H^*_T (M^H) \cong H^*_{T'} (M^H) \otimes H^*_H
(M^H).$ Note that  $H^*_H (M^H)\cong  H^* (BH)$ (it is enough to remember that $
H^*_H (M^H)$ is defined to be the ordinary cohomology ring  $H^* (EH \times_H
M^H)$ and to note that $H$ acts trivially on its fixed point set $M^H$).
Therefore $$ H^*_T (M^H)\cong H^*_{T'} (M^H) \otimes  H^* (BH).$$
It follows that the restriction to $M^H$ of any class $a \in H^*_T (M)$
decomposes as $a_{|_{M^H}} = \sum_{i\geq 0} a_i \otimes u^i$ where $u$ is the
positive generator of $H^*(BH)$ and the $a_i$ are elements in $ H^*_{T'}
(M^H).$ 
\begin{prop} \label{lequiv} \cite{Martin} With the notation above,
$$\lambda_H(a)= k(M^H_i) \sum_{i \geq 0} a_i \smile s^w_{i-\rho +1}$$
where $k(M^H_i)$ is the greatest common divisor of the weights of the $H$-action
on the fibers of the normal bundle $\nu M^H_i \rightarrow M^H_i,$ $s_j^w$
denotes the $j$-th $T'$-equivariant Segre class of $(\nu M^H, H)$ and $\rho$ is
the function (constant on the connected components of $M^H$) such that $2 \rho =
\mathrm{rank}(\nu M^H).$
\end{prop}
The next result relate integration over the symplectic quotients respectively by
a non abelian group $G$ and by a maximal torus in $G$. 

Let $G$ be a connected compact Lie group which acts on the smooth manifold $M$
in a Hamiltonian way, with associated moment map $\mu_G$. Let $T$ be a maximal
subtorus in $G.$ The restriction of the $G$-action defines a Hamiltonian
$T$-action on $M$  (with associated moment map $\mu_T$). There is a natural
restriction map $r_T^G : H^*_G(M) \rightarrow H^*_T(M) $ between the equivariant
(with respect to $G$ and $T$) cohomology rings.  To fix the notation, 
$\C^m_{(w)}$ denotes the complex space $\C^m$ endowed of the $S^1$-action with 
weight $w$ and $\underline{\C}_{(w)}:= M \times \C_{w}$ is the total space of an
equivariant line bundle over $M.$
\begin{theorem}\label{int.formula} \emph{Equivariant integration formula.}
\cite{M2} \newline
For all $a \in H^*_G(M),$ $$\int_{M/\!\!/G} k_G (a) = \frac{1}{| W |}
\int_{M/\!\!/T} k_T \Big(r_T^G(a) \smile \prod_{\alpha \in \Delta}
c_1^T(\underline{\C}_{\alpha}) \Big), $$
where $|W|$ is the order of the Weyl group of $G$ and $\Delta $ is the set of
roots of $ G.$
\end{theorem}

\subsection{The Volume Theorem}\label{volu}

\begin{theorem}\label{volume}
For generic $r \in \R^n_+,$
$$ \vol \, M(r)=- \, \frac{(2 \pi)^{n-3}}{2(n-3)!}\,  \sum_{I \text{long}}
(-1)^{n-|I|} \, \epsilon_I(r)^{n-3} $$ where $ I \subset \{ 1,
\ldots, n\}$ is long if and only if $  \epsilon_I(r)= \sum_{i \in I}r_i- 
\sum_{i
\in I^c}r_i  >0$.
\end{theorem}
\begin{proof}

The first step in the proof is to apply Theorem \ref{int.formula} and write the
volume of $M(r)$ as
$$ \vol \, M(r) = \frac{1}{2} \int_{\mathcal{S}_r/\!\!/S^1}
k_{S^1}(r^{SO(3)}_{S^1}(a) \smile c_1^{S^1}(\underline{\C}_{(1)}) \smile
c_1^{S^1}(\underline{\C}_{(-1)}))$$
where $a \in H^*_{SO(3)}(\mathcal{S}_r)$ is such that $k_{SO(3)}(a) $ is the
volume form on $\mathcal{S}_r /\!\!/ SO(3) =M(r)$ and $S^1$ is a (arbitrarily
chosen) maximal subtorus of $SO(3).$ (We have already entered in the formula
that the Weyl group of $SO(3)$ is $\mathbb{Z}/2\mathbb{Z}$ and that  the set of
roots of $SO(3)$ is $\{ \pm 1\}.$)

The second step is to apply the localization Theorem \ref{localization} to
localize the calculation of the integral above to data associated to the fixed
points set of the $S^1$-action. 

Remember that  the symplectic structure on $\mathcal{S}_r$ is defined by the
$2$-form  $\omega= \sum_{j=1}^n \frac{1}{r_j} p_j^* \omega_j,$ where 
$p_j:\mathcal{S}_r \rightarrow S^2_{r_j}$ is the natural projection on the
$j$-th factor  and $\omega_j $ is the volume form on the sphere $S_{r_j}^2.$ It
is a calculation to check that, if $\alpha $ is the volume form on the unit
sphere and $\omega_{FS}$ is the Fubini--Study form on $\C\P^1 \simeq S^2,$ then
$\omega_{j}= r_j \alpha = 2 r_j \omega_{FS}.$ 

On each sphere consider the line bundle $\mathcal{O}(2r_j) \rightarrow S^2_j.$
The tensor product of their pullbacks $p_j^* \mathcal{O}(2r_j)$ defines on
$\mathcal{S}_r$ the line bundle $\mathcal{L} := \mathcal{O}(2r_1) \boxtimes
\cdots \boxtimes \mathcal{O}(2r_n)$ (known in literature as the  prequantum line
bundle of $\mathcal{S}_r$).  Observe that $\omega_{FS}$ is the first Chern class
of $\mathcal{O}(1),$ precisely $ [\frac{\omega_{FS}}{2\pi}] = c_1
(\mathcal{O}(1)).$ It follows by the definition of the symplectic form $\omega$
on $\mathcal{S}_r$ that $$ \Big[\frac{\omega}{2\pi} \Big] =
c_1(\mathcal{O}(2r_1) \boxtimes \cdots \boxtimes \mathcal{O}(2r_n) )= c_1
(\mathcal{L}).$$

The construction above is well defined just for integral $r_1, \ldots, r_n,$ so
let us restrict to the case $r \in \mathbb{Z}^n_+$ and prove the stated result
for the volume of $M(r).$  Then, for each $\lambda \in \R^+,$ we get the volume
of $M(\lambda r)$ by rescaling. Indeed, $\vol \, M(\lambda r) = (\lambda)^{n-3}
(\vol \, M(r)),$ thus the formula holds also for rational $r_i.$ Finally, by
density, the result extends to $r \in \R^n_+.$ 

Let $a$ be the $(n-3)$-th power of the first equivariant Chern class
$c_1^{SO(3)}(\mathcal{L})$ of the prequantum line bundle $\mathcal{L}$
(normalized by a factor $\frac{(2 \pi)^{n-3}}{(n-3)!}$ ). Then its image $k(a)$
through the Kirwan map $k:H^*_{SO(3)}(\mathcal{S}_r) \rightarrow
H^*(\mathcal{S}_r/\!\!/SO(3))$ is the volume form on $M(r):$ $$ \vol \, M(r)=
\frac{(2 \pi)^{n-3}}{(n-3)!} \int_{M(r)} k(c_1^{SO(3)}(\mathcal{L})^{n-3}).$$
We now apply the equivariant integration formula (Theorem \ref{int.formula}).
The restriction  $r^{SO(3)}_{S^1}$ maps   $c_1^{SO(3)}(\mathcal{L})^{n-3}$ to
$c_1^{S^1}(\mathcal{L})^{n-3},$ thus
$$  \vol \, M(r)= \frac{1}{2} \frac{(2 \pi)^{n-3}}{(n-3)!}
\int_{\mathcal{S}_r/\!\!/S^1}  k_{S^1}( c_1^{S^1}(\mathcal{L})^{n-3} \smile
c_1^{S^1}(\underline{\C}_{(1)}) \smile c_1^{S^1}(\underline{\C}_{(-1)}))$$ and
the first step is done.

In order to apply the localization Theorem \ref{localization} we make an
explicit choice of a maximal subtorus $S^1 \subset SO(3):$ let $S^1$ be the
subgroup that acts on each sphere by rotation along the $z$-axis.
This action is Hamiltonian with moment map the height function $$
\begin{array}{rcl}
 \mu:S^2_{r_j} & \rightarrow & \mathfrak{s}^1 \simeq \R \\
 e_j=(x_j, y_j, z_j) & \mapsto &ht( e_j)= z_j.
\end{array}$$ Note that the fixed points of this action are the north pole $N_j$
and the south pole $S_j$ and the image $\mu (S^2_{r_j}) $ is the segment $[\mu
(S_j), \mu (N_j)]= [-r_j, r_j]$ (in agreement with the convexity Theorem).

These observations extend easily to the product manifold $\mathcal{S}_r:$
consider on $\mathcal{S}_r$ the circle action by rotation around the $z$-axis of
each sphere.
This action is Hamiltonian and, by linearity, has moment map the sum of the
heights, i.e. $\mu (e_1, \ldots, e_n) = \sum_{j=1}^n z_{j}.$ 

A point $(e_1, \ldots , e_n)$ is fixed by this action if and only if $e_{j} \in
\{ N_j, S_j\}$ for each $j\in \{ 1, \ldots, n\},$ and these points are isolated.
For these points we introduce a more handy notation: let $I $ be any subset of $
\{ 1, \ldots, n\}.$ We define $f_I$ to be the point  $(e_1, \ldots , e_n)\in
\mathcal{S}_r$  such that $e_j $ is a north pole if $j\in I$ and a south pole if
$j \in I^c$. Thus all the fixed points are an $f_I$ for some index set $I$ and
$$ \mu(f_I) =  \sum_{i \in I}r_i -  \sum_{i \in I^c}r_i = \epsilon_I(r).$$
\begin{remark}
Note that $ \epsilon_I(r)  \ne 0$ for all $I$ because we assumed $r$ generic.
This implies that $0$ is a regular value of the moment map $\mu.$ In fact $d_x
\mu $ is identically $0$ if and only if $x= f_I:$ for each tangent vector $v=
(v_1, \ldots ,v_n)\in T_x \mathcal{S},$ $d_x \mu (v) = \sum_j \zeta_j$ where
$\zeta_j$ is the third component of $v_j.$ So $d_x \mu \equiv 0 \iff \zeta_j=0
\quad \forall j \iff x= f_I \quad \textit{for some} \quad I.$
\end{remark}
From the Atiyah and  Guillemin-Sternberg convexity Theorem, the image $\mu
(\mathcal{S}_r)$ is the convex hull of the points $\mu (f_I),$ i.e. $$\mu
(\mathcal{S}_r) = \Big[-\sum_{i=1}^n r_i,\sum_{i=1}^n r_i \Big].$$  The idea is
now to apply Theorem \ref{localization} to calculate the volume of
$\mathcal{S}_r/\!\!/S^1.$ Choose $p_0=0 $ and $p_1 > \sum_{i=1}^n r_i,$ so that
$\mu^{-1}(p_1)$ is empty, this implies that the integral over
$\mathcal{S}_r/\!\!/S^{1}(p_1)$ is zero and 

$$ \int_{\mathcal{S}_r/\!\!/S^{1}}  k(\tilde{a}) =\sum \int_{X^{H} /\!\!/T (q)}
k_q(\lambda_H(\tilde{a}_{|_{X^H}}))$$
where the sum is made over the walls $\mu(X^{H_i})$ that the path  $Z=[0, p_1]
\subset \R$ crosses at $q_i.$

Moreover note that  the walls in  $\mu (\mathcal{S}_r)$ are just the points 
$\mu (f_I),$ and that the path $Z$ crosses  the walls  $\mu (f_I)$ only for
those $I$ such that  $ \epsilon_I(r) >0.$ Let $\mathcal{I}$ be the family of all
these index sets $I$. 
Since that the quotient spaces $X^H /\!\!/T (q)$ are just points, we obtain $$
\int_{\mathcal{S}_r/\!\!/S^{1}}  k(\tilde{a}) = \sum_{I \in \mathcal{I}}
k_I(\lambda_I(\tilde{a}_{|_{f_I}})).$$

Now we will study the normal bundle $\nu f_I$ in order to work out the necessary
details to use the equivariant description  of $\lambda_{f_I}$ (see Proposition
\ref{lequiv}).

The $f_I$'s are points thus for each $I$ the normal bundle $\nu f_I$ is the
direct sum of copies of $T_{N_j} S_{r_j}^2$ and   $T_{S_j} S_{r_j}^2.$ Precisely
$$ \nu f_I \simeq \C^{|I|}_{(1)} \oplus \C^{n-|I|}_{(-1)}.$$
The equivariant Segre classes that appear in Proposition \ref{lequiv} formally
lie in $H^*_{T/H}(f_I),$ where $H \simeq S^1 $ is the subgroup of $T$ associated
to the wall $\mu(f_I);$ in our case $T$ is $S^1$ itself, so $s^w(\nu f_I)$ lies
in the de Rham cohomology ring $H^*(f_I).$
The  bundle $\nu f_I$ has rank one, and the $j$-th Chern classes $c_j(\C_{(\pm
1)})$ are zero for each $I$ and $j$ (because, for each $I,$ $\nu f_I$ is a line
bundle over a point). Then $$c^w(\nu f_I)=  (-1)^{n- |I|}$$ and $$ s^w_j(\nu
f_I) \left\{
\begin{array}{ll} 
(-1)^{n- |I|} & j=0,\\
0 & \textrm{otherwise}.\\
\end{array} \right.$$
Moreover, the greatest common divisor $k(f_I )= 1$ for each $I$ because the
weights are all $\pm 1.$

We have now all the ingredients to apply the equivariant formula in Proposition
\ref{lequiv} and calculate $\lambda_{I}(\tilde{a}_{|_{f_I}}),$ with $\tilde{a}=
c_1^{S^1}(\mathcal{L})^{n-3} \smile c_1^{S^1}(\underline{\C}_{(1)}) \smile
c_1^{S^1}(\underline{\C}_{(-1)}).$

From the construction of the line bundle $\mathcal{L}$ we made above, it follows
that $\mathcal{L}_{|_{f_I}} = \C_{(\epsilon_I(r))}$ where again $I$ is the index
set that ``detects'' the north poles. Thus $$c_1^{S^1}(\mathcal{L})_{|_{f_I}}=
(\epsilon_I(r)) u, $$ where $u$ is the positive generator of the equivariant
cohomology of a point $H^*_{S^1}(f_I).$ Similarly,
$$c_1^{S^1}\Big(\underline{\C}_{(1)}\Big)_{|_{f_I}}= u, \quad \quad
c_1^{S^1}\Big(\underline{\C}_{(-1)}\Big)_{|_{f_I}}= -u.$$ So $$
\tilde{a}_{|_{f_I}} = -(\epsilon_I(r))^{n-3} u^{n-1}$$ and
$$\lambda_I(\tilde{a}_{|_{f_I}})= - (-1)^{n- |I|} (\epsilon_I(r))^{n-3}
u^{n-1}.$$
To finish the proof we should now apply the Kirwan map $k_q: H^*_{T/H} (X^H)
\rightarrow H^*(X^H/ \!\! /T (q))$ as in Theorem \ref{localization}. Since in
our case $T$ is $S^1$ itself and the fixed points sets $X^H$ are the $f_I$'s,
the map $k_q: H^*(f_I) \rightarrow H^*(f_I)$ is the identity. Thus, summing on
all the admissible $I$, the result follows.

\end{proof}

\subsection{Examples}\label{esvolume}
Let $\Delta^{\!0} $ be the chamber in $\Xi\in \R^5$ determined by its collection
of short sets
\begin{align*} \mathcal S(\Delta^{\!0}) =& \big\{ \{i\} \mid i=1,\ldots, 5\}
\cup \{ \{j,k\} \mid
 j,k= 1,2,4,5 \big\} \\ &
\cup\big\{ \{i, j,k\} \mid  i,j,k= 1,2,4,5\big\}.
\end{align*}
Consequently the collection of $r^0$-long sets is  
$$ \mathcal{I} (r^0) = \big\{ \{3,j\},\{3,j,k\} : j,k= 1,2,4,5 \big\} \cup
\big\{I\subseteq \{1, \ldots, 5 \} : |I|= 4,5  \big\}.$$
Note that the chamber $\Delta^{\!0}$ is not empty, for example $\frac{1}{7}
\big( 1,1,3,1,1 \big) \in \Delta^{\!0}$.
Then, by Theorem \ref{volume}, it follows (by plain computation) that the volume
of the associated
symplectic quotient $M(r^0)$ is 
$$ \vol \, M(r^0) = 2 \pi^2 (r_1^0 +r_2^0 -r_3^0
+r_4^0 +r_5^0 )^2.$$
Because the perimeter $\sum_{i=1}^n r_i = 1$ is fixed on $\Xi,$  one also
obtains 
$ \vol \, M(r^0) = 2 \pi^2 (1-2r_3^0 )^2$.

Consider now the adjacent chamber $\Delta^{\! 1}$ characterized by 
\begin{align*} \mathcal S(\Delta^{\!1}) = &\big\{ \{i\} \mid i=1,\ldots, 5\}
\cup \{ \{j,k\} \mid
 j,k= 1,2,4,5 \big\} \cup\{ 1,3\} \\ &
\cup\big\{ \{1, j,k\} \mid  j,k= 2,4,5\big\}.
\end{align*}
Also the chamber $\Delta^{\! 1}$ is not empty, as for example the lengths vector
$\frac{2}{11} \big( \frac{1}{2},1,2,1,1 \big)$ is in $\Delta^{\! 1}$.
The closures of $\Delta^{\!0}$ and $\Delta^{\!1}$ intersect in the wall of
equation $\epsilon_{\{1,3\}} (r)= \epsilon_{\{ 2,4,5\}} (r) =0$.
This means that $\{1,3\}$ is $r^0$-long and $r^1$-short, while its complement
$\{2,4,5\}$ is $r^1$-long and $r^0$-short. This is the
only difference between  in the collections of $r^0$-long sets and  $r^1$-long
sets.
Hence, applying Theorem \ref{volume}, we get 
$$\vol \, M(r^1) = 4 \pi^2 r_1^1 ( r_2^1 - r_3^1 + r_4^1 +r_5^1 ).$$
Again, using the fixed perimeter condition, one obtains
$\vol \, M(r^1)= 4 \pi^2 r_1^1 (1 -r_1^1 -2r_3^1)$.

\section{Crossing the Walls}\label{crossing the walls}
In this Section we 
explicitly describe how the diffeotype of the manifold $M(r)$ changes as $r$
crosses a wall in  $\Xi= \mu_{_{U(1)^n}}(Gr(2,n))$. 

The chambers $\Delta^{\! i}$ of regular values in the convex polytope $\Xi$ are
convex polytopes themselves. They are separated by walls, i.e. by the images $
\mu_{_{U(1)^n}} (Gr(2,n)^{S^1})$ of the sets of points fixed by the circle
subgroups of $U(1)^n.$ 
For $r^0$ and $r^1$ in different chambers $\Delta^{\! 0} $ and $\Delta^{\! 1},$
the symplectic reductions $M(r^0)$ and $M(r^1)$ are related by a birational map
that can be described in terms of blowing up and down submanifolds. This follows
from a general construction due independently to Brion-Procesi \cite{brion} and
to Guillemin--Sternberg \cite{gs89}. In this Section we show that these
submanifolds are resolutions of the singularity corresponding to the degenerate
polygon in the singular quotient $M(r^c)$ (where $r^c$ is the wall-crossing
point) and characterize them in terms of polygon spaces of lower dimension.

Through all the paper we assume a single wall-crossing, meaning that the wall
crossing point is not on a intersection of walls, but lies on one and only one
wall. This also implies that the quotient $M(r^c)$ has only one critical point.
The assumption is not restrictive since any non single wall-crossing can be
decomposed in a finite number of subsequent single wall-crossing, cf.
\cite{gs89}.

In \cite{gs89} Guillemin and Sternberg give a thorough analysis of  wall
crossing problems relative to (quasi-free) $S^1$-actions. They also point out
that their construction can be made $H$-equivariant, when $H$ is any compact
group commuting with the $S^1$-action. This is our case: in fact we will first
perform the symplectic reduction by a complement $H$ of the $S^1$ associated to
the wall and then apply the analysis as in \cite{gs89} to the remaining
$S^1$-action. Still there is a small subtlety here, since the action of $U(1)^n$
is not effective. 

In this Section we prove the following Theorem \ref{WCtheo}. Before stating
that, let us introduce some notation: consider $r^0$ and $r^1$ regular values of
$\mu_{U(1)^n}$ lying in different chambers, $\Delta^{\! 0}$
and $\Delta^{\! 1}$ respectively, separated by the wall of equation 
\begin{equation}\label{muro1}
\varepsilon_{I_p} (r)=0.
\end{equation}
Assume also that the lengths vectors $r^0 \in \Delta^{\!0} $ and $r^1 \in
\Delta^{\!1}$ satisfy
\begin{equation}\label{direction}
\varepsilon_{I_p} (r^0)> 0 \quad \textnormal{and} \quad \varepsilon_{I_p} (r^1)
<0 
\end{equation}
and call $W_{I_p}$ the wall of equation \eqref{muro1} together with the
wall-crossing direction from $\Delta^0$ to $\Delta^1$.
Moreover, for any index set $I \subset \{ 1, \ldots n\}$, let $M_{I}(r)$ be the
(eventually empty) submanifold of $M(r)$ of those polygons such that the edges
$e_i$, for $i \in I$, are positive proportional to each other. Precisely 
$$ M_{I} (r) := \widetilde{M}_{I} (r)/ SO(3)$$
where
\begin{equation}\label{eq:Mtilde}
\widetilde{M}_{I} (r) := \{ (e_1, \ldots, e_n) \in \mathcal{S}_r \mid
\sum_{j=1}^n e_j =0, \, e_i= \lambda_k e_k, \, \, \forall i,k \in I, \, 
\lambda_k \in \R_+ \}. 
\end{equation}

\begin{theorem}\label{WCtheo}
Let the lengths vector $r$ cross a wall $W_{I_p} $ in $\Xi$ as above. Then the
diffeotype of the moduli space of polygons $M(r)$ changes by blowing up
$M_{I_p^c}(r^0) \simeq \C\P^{p-2}$ and blowing down the projectivized normal
bundle of $M_{I_p} (r^1) \simeq \C\P^{q-2}.$

The polygon spaces $M_{I_p}(r^1)$ and $M_{I_q}(r^0)$ are resolutions of the
singularity corresponding to the degenerate polygon $[P^c]$ in $M(r^c),$ and
both are dominated by the blow up $\widetilde{ M}$ of $M(r^c)$ at the singular
point, with exceptional divisor $\C\P^{p-2} \times \C\P^{q-2}$.
\end{theorem}

\begin{proof}
Let $H':= \{ \textnormal{diag} (e^{i \beta_1}, \ldots , e^{i \beta_{n-1}}, 1)
\mid e^{i \beta_j} \in S^1 \quad\forall j= 1, \ldots, n-1 \} \subset U(1)^n$ be
a complement of the diagonal circle $\{ \textnormal{diag} (e^{i \theta}, \ldots,
e^{i \theta})\} \subset U(1)^n$.
The group $H'$ acts effectively on $Gr(2,n)$ by restriction of the
$U(1)^n$-action with associated moment map 
\begin{displaymath}
\begin{array}{rcl}
\mu_{H'}: Gr(2,n) & \rightarrow & \R^{n-1}\\
(a,b) & \mapsto & \frac{1}{2} (|a_1|^2 +|b_1|^2, \ldots , |a_{n-1}|^2
+|b_{n-1}|^2).\\
\end{array}
\end{displaymath}
The moment polytope $ \mu_{H'}(Gr(2,n)) $ is the image of $\Xi$ via the
projection map $\R^n \rightarrow \R^{n-1}$ that drops the last coordinate.
Since the action of $U(1)^n$ is not effective (the diagonal circle fixes every
point), one can easily see that both the quotients 
$$ 
\mu_{U(1)^n}^{-1} (r_1, \ldots, r_n) / U(1)^n \quad \textnormal{and} \quad
\mu_{H'}^{-1} (r_1, \ldots, r_{n-1}) / H'
$$
are diffeomorphic, and each is the moduli space of polygons $M(r)$. 
Note in particular that $r_n$ is
uniquely determined by $r_1, \ldots, r_{n-1}$. In other words, if $(r_1, \ldots,
r_n)$ are coordinates in $\Xi$, then the coordinates $ r_1, \ldots, r_{n-1}$ on
the projected polytope $ \mu_{H'}(Gr(2,n)) $  satisfy
\begin{displaymath}
r_n = 1 - \sum_{i=1}^{n-1} r_i.
\end{displaymath}
It follows that the wall $\epsilon_{I_p}(r) = 0 $ is mapped in $
\mu_{H'}(Gr(2,n)) $ onto the wall of equation
\begin{equation}\label{muro}
 \sum_{i \in I_p} r_i = \frac{1}{2}.
\end{equation}
In particular external walls satisfy $2 \sum_{i \in I_p} r_i = 1$ for $I_p$ of
cardinality $1$ or $n-1$. 
Note that whenever \eqref{muro} holds, then the condition $\sum_1^n r_i=1$
implies $2 \sum_{i \in I_p^c} r_i = 1$, where $I_p^c$ is the complement of $I_p$
in $ \{ 1, \ldots, n \}$. Therefore there exists at least one index $i \in I_p$
and at least one index $j \in I_p^c$ such that $r_i \neq 0 $ and $r_j \neq 0$.
Since $M(r)$ is symplectomorphic to $M(\sigma (r))$ for any permutation $\sigma$
of the $n$ edges, it is not restrictive to assume $I_p = \{ 1, \ldots p\}$ and
$r_1 \neq 0$ as well as $r_n \neq 0$.

The orientation of the circle associated to $W_{I_p}$ is determined accordingly
with the wall-crossing direction. This means that among the directions 
$$v^{\pm} = \pm (\,  \underbrace{-1,\ldots,-1}_p, 0, \ldots, 0)$$
normal to the wall \eqref{muro}, we wish to choose the one that has positive
inner product with the vector $(r^1- r^0) \in \R^{n-1}$. This is the case for
$v^-$ (it follows from assumption \eqref{direction}), and therefore the circle
associated to $W_{I_p}$ is 
$$ S^1:= \{ \textnormal{diag} (\,  \underbrace{e^{-i \theta}, \ldots, e^{-i
\theta}}_p, 1, \ldots, 1)\} \subset H' \subset U(1)^n.$$
Let $r^c$ be the wall-crossing point and let 
$$H:= \{ \textnormal{diag} (1, e^{i \theta_2}, \ldots, e^{i \theta_{n-1}}, 1)
\mid e^{i \theta_j} \in S^1 \, \forall j= 2, \ldots, n-1 \}$$ be a complement of
$S^1$ in $H'$. The group $H$ acts on $Gr(2,n)$ in a Hamiltonian fashion with
associated moment map 
$$ \mu_H (a,b) = \frac{1}{2} (|a_2|^2 +|b_2|^2, \ldots , |a_{n-1}|^2
+|b_{n-1}|^2).$$
We now analyze the $S^1$-action on $\mu_H^{-1} (r_2^c, \ldots, r_{n-1}^c) / H$.
In particular, if $\mu_{S^1} $ is the moment map for the residual $S^1$-action
on $\mu_H^{-1} (r_2^c, \ldots, r_{n-1}^c) / H$, we will describe the singular
reduced space $\mu_{S^1}^{-1}(0) / S^1 $ as in \cite{guillemin}, obtaining also
the two resolutions of the singularity in Theorem \ref{WCtheo}. 

To this aim, note that the fixed points set of the $S^1$-action consists only of
the point $[ P^c]$:
\begin{equation} \label{P}
\Big( \mu_H^{-1} (r_2^c, \ldots, r_{n-1}^c) / H \Big)^{S^1}= [ P^c] := {\small 
\left( \begin{array}{cc} 
\sqrt{1- 2 \sum_2^p r_i^c}& 0\\
\sqrt{2 r_2^c} & 0\\
\vdots& \vdots\\
\sqrt{2 r_p^c}& 0\\
0 & \sqrt{2 r_{p+1}^c}\\
\vdots& \vdots \\
0 & \sqrt{2 r_{n-1}^c}\\
0& \sqrt{1- 2 \sum_{p+1}^{n-1} r_i^c}
\end{array} \right)}.
\end{equation}
To prove \eqref{P} note that an element $(a,b) \in \mu_H^{-1} (r_2^c, \ldots,
r_{n-1}^c) / H$ is fixed by the $S^1$-action if and only if 
$$
\left\{ \begin{array}{ll}
a_i=0, & \forall i=p+1, \ldots, n, \\
b_j=0 & \forall j=1, \ldots, p.
\end{array} \right.$$
It then follows from the moment map conditions that $|a_i| = \sqrt{2 r_i^c}$ for
all $i=2, \ldots, n$ and $|b_j| = \sqrt{2 r_j^c}$ for all $j=p+1, \ldots, n$.
Recalling that $\sum_1^n |a_i|^2= 1$ and $\sum_1^n |b_i|^2= 1$ since $(a,b) \in
St_{2,n}$, we get that 
$$a_1= \sqrt{1- 2 \sum_2^p r_i^c} \,\, e^{i \phi_1} \quad \textnormal{and} \quad
b_n = \sqrt{1- 2 \sum_{p+1}^{n-1} r_i^c} \,\, e^{i \phi_n}$$ for some $\phi_1,
\phi_n \in [0, 2\pi[ $. Modulo the $SU(2)$-action, we can then take $a_1$ and
$b_n$ to be real.
Now modulo the $H$ action, we can take $e^{-i \phi_1} a_i$, $i=2, \ldots, p$,
and $e^{i \phi_n} b_j$, $j=p+1, \ldots, n-1$, to be real as well, and \eqref{P}
follows. 

In a neighborhood $\mathcal U \subset \mu_H^{-1} (r_2^c, \ldots, r_{n-1}^c) / H$
of the fixed point $[P^c]$ we give a local system of coordinates $(w_2, \ldots,
w_p, z_{p+1}, \ldots, z_{n-1}) \in \C^{n-2}$, centered at $[P^c]$, such that
\begin{equation}\label{coordlocali}
P= {\small 
\left( \begin{array}{cc} 
l_1& 0\\
l_2 & w_2\\
\vdots& \vdots\\
l_p& w_p\\
z_{p+1} & m_{p+1}\\
\vdots& \vdots \\
z_{n-1} & m_{n-1}\\
0& m_n
\end{array} \right)} \quad \forall P \in \mathcal U
\end{equation}
with $l_j$ and $m_k$ real functions of $(w,z)$ and of the wall-crossing value
$r^c$.
These local coordinates can be determined as follows.
Given $(a,b) \in \mathcal U$ consider a non-zero minor, for example the one
formed by the first and the last row (this clearly does not vanish for $\mathcal
U$ neighborhood of $[P^c]$ small enough). Then using the $U(2)$ action we can
rewrite $(a,b)$ as 
\begin{equation}\label{intornoP}
{\small 
\left( \begin{array}{cc} 
\tilde{l}_1& 0\\
\tilde{z}_2 & \tilde{w}_2\\
\vdots& \vdots\\
\tilde{z}_{n-1} & \tilde{w}_{n-1}\\
0& \tilde{m}_n
\end{array} \right)}
\end{equation}
where $\tilde{l}_1 = |a_1|^2 + |b_1|^2$, $\tilde{m}_n =|a_n|^2 + |b_n|^2$ and
\begin{equation}\label{zw}
\tilde{z}_i = \frac{a_ib_n-a_nb_i}{a_1b_n-a_nb_1}, \quad \tilde{w}_i =
\frac{a_1b_i-a_ib_1}{a_1b_n-a_nb_1}. 
\end{equation}
Writing $\tilde{z}_j= |\tilde{z}_j| e^{i \tilde{\theta}_j}$ for $j=2, \ldots,p$
and $\tilde{w}_k= |\tilde{w}_k| e^{i \tilde{\theta}_k}$ for $k=p+1, \ldots,n-1$,
one can see that, modulo the $H$-action, \eqref{intornoP} becomes
\eqref{coordlocali}, where
$$ w_j=  e^{-i \tilde{\theta}_j} \tilde{w}_j, \quad l_j = |\tilde{z}_j|= \sqrt{2
r_j^c - |w_j|^2} \quad \forall j=2, \ldots, p,$$
$$z_k = e^{-i \tilde{\theta}_k} \tilde{z}_k, \quad m_k = |\tilde{w}_k| = \sqrt{2
r_k^c - |z_k|^2}\quad \forall k=p+1, \ldots, n-1$$
and consequently
$$ l_1 = \Big(1 - 2 \sum_{j=2}^p r_j^c + \sum_{j=2}^p |w_j|^2 -
\sum_{k=p+1}^{n-1} |z_k|^2 \Big)^{1/2};$$
$$ m_n= \Big( 1-  2 \sum_{k=p+1}^{n-1} r_k^c -\sum_{j=2}^p |w_j|^2 +
\sum_{k=p+1}^{n-1} |z_k|^2 \Big)^{1/2}.$$
In such a neighborhood $\mathcal U$ of $[P^c] $ the action of $S^1$ is then 
$$ \textnormal{diag}(e^{-i \theta}, \ldots, e^{-i \theta}, 1, \ldots, 1) \cdot
(w_2, \ldots, w_p, z_{p+1}, \ldots, z_{n-1})$$ 
$$= (e^{-i \theta} w_2, \ldots, e^{-i \theta} w_p, e^{i \theta} z_{p+1}, \ldots
, e^{i \theta} z_{n-1})$$
with associated moment map
$$ \mu_{S^1} (w_2, \ldots, w_p, z_{p+1}, \ldots, z_{n-1}) = \frac{1}{2} \Big(
-\sum_{j=2}^p |w_j|^2 + \sum_{k=p+1}^{n-1} |z_k|^2 \Big).$$

The critical level set $$ \sum_{j=2}^p |w_j|^2= \sum_{k=p+1}^{n-1} |z_k|^2 $$ is
a conic subset of $\C^{n-2}$. Precisely, it is the cone over the product of the
two spheres $S^{2p-3}= \{ w_j \mid \sum_{j=2}^p |w_j|^2=1 \}$ and  $S^{2q-3}= \{
z_k \mid \sum_{k=p+1}^{n-1} |z_k|^2 =1\}$, where $q:=n-p$. The action of $S^1$
on this product is free so the orbit space $W= (S^{2p-3} \times S^{2q-3})/ S^1$
is a compact manifold and the quotient $\mu_{S^1}^{-1}(0) / S^1$ in the
neighborhood $\mathcal U $ of $[P^c]$ looks like a cone $C_W$ over $W$ with
vertex at $\{0\} = [P^c]$.

From the action $(w_2, \ldots, w_p) \mapsto (e^{-i \theta} w_2, \ldots, e^{-i
\theta} w_p)$ of $S^1$ on $S^{2p-3} $ one gets the Hopf fibration $$\pi:
S^{2p-3} \rightarrow \C\P^{p-2}.$$
Since $S^1$ acts also on $S^{2q-3}$, one can consider the associated bundle
$$ (S^{2p-3} \times S^{2q-3}) / S^1 \rightarrow \C\P^{p-2}$$
obtaining a description of $W$ as a fiber bundle over $\C\P^{p-2}$. Reversing
the roles of $p$ and $q$ we get $W$ as a fiber bundle over $\C\P^{q-2}$.
Associated with these two description of $W$ we obtain the desingularizations of
$M(r^c)$ as in Theorem \ref{WCtheo}.
In fact, since the action of $S^1$ on $S^{2q-3}$ extends to a  linear action on
$\C^{q-1}$ one can form the associated vector bundle
\begin{equation}\label{W-}
\xymatrix{W_- := \hspace{-1cm} & S^{2p-3} \times_{\pi} \C^{q-1} \ar[d]\\
  & \C\P^{p-2}.}
\end{equation}
On this bundle there is a blowing down map $\beta_- : W_- \rightarrow C_W$,
$$ \beta_- (w_2, \ldots, w_p, z_{p+1}, \ldots, z_{n-1})= \Big(
\sqrt{\sum_{i=p+1}^{n-1} |z_i|^2} \,\,w_2, \ldots, \sqrt{\sum_{i=p+1}^{n-1}
|z_i|^2} \,\,w_p, z_{p+1}, \ldots, z_{n-1}\Big)$$
and an embedding of $\C\P^{p-2}$ as the zero section of the bundle \eqref{W-}:
\begin{displaymath}
\begin{array}{rcl}
\iota: \C\P^{p-2} & \rightarrow& W_- \\
{[w_2, \ldots, w_p]} & \mapsto & (w_2, \ldots, w_p, 0 \ldots, 0).\\
\end{array}
\end{displaymath}
Moreover, the image of $\C\P^{p-2}$ in $W_-$ gets blown down to $\{0\} \in C_W$,
and $\beta_-$ maps $W_- \setminus \C\P^{p-2}$ diffeomorphically onto $C_W
\setminus \{0\}$.

On the other hand, reversing again the roles of $p$ and $q$ one obtains a
desingularization $W_+$ of $M(r^c)$, where $W_+ $ is now the bundle 
\begin{equation}\label{W+}
\xymatrix{W_+ := \hspace{-1cm} & \C^{p-1} \times_{\pi} S^{2q-3} \ar[d]\\
  & \C\P^{q-2}}
\end{equation}
with associated blowing down map $\beta_+ : W_+ \rightarrow C_W$ and embedding
$\iota : \C\P^{q-2} \rightarrow W_+$ as the zero section.

Via these two desingularizations we obtain our description of the birational map
from $M(r^0)$ to $M(r^1)$ with $r^0$ and $r^1$ as in  Theorem \ref{WCtheo}.

In fact for $-\sum_{j=2}^p |w_j|^2 + \sum_{k=p+1}^{n-1} |z_k|^2 = - \epsilon$
the orbit space $\mu_{S^1}^{-1}(- \epsilon) / S^1 $ is identical topologically
with $W_-$. Note that $r_1$ is then uniquely determined 
$$ r_1 =l^2_1 = 1- \sum_{j=2}^p r_j^c + \epsilon$$
and consequently 
$$ r_n =  1- \sum_{k=p+1}^{n-1} r_k^c - \epsilon.$$
This means that $r^0 := (r_1, r_2^c, \ldots, r_{n-1}^c, r_n )$, with $r_1$ and
$r_n$ as above, satisfies $\epsilon_{I_p}(r^0) > 0$.
Using the right hand side of diagram \eqref{beauty} we can give a geometric
characterization of the $\C\P^{p-2}$ that is blown down by the map $\beta_-$,
describing it as a lower dimensional polygon space.

As seen above, $\C\P^{p-2}$ is embedded in $W_-$ as the zero section with
respect to the local coordinates $ (w,z)$, i.e. 
\begin{equation}\label{eq:CP2}
\C\P^{p-2} = \{ (w_1, \ldots, w_p, z_{p+1}, \ldots, z_{n-1}) \mid z_k =0 \,\,
\forall \, k= p+1, \ldots, n-1 \}.
\end{equation}
From \eqref{zw} it follows
\begin{equation}\label{zk0}
z_k=0 \iff a_kb_n - a_n b_k=0.
\end{equation}
We now want to describe $\C\P^{p-2}$ performing the reductions as the right hand
side of the diagram \eqref{beauty}.
To this aim, start from $(a,b) \in p^{-1}(\mu_{U(1)^n}^{-1} (r)) \subset
St_{2,n}$ satisfying \eqref{zk0} for every $k= p+1, \ldots, n-1$ and consider it
as an element in $\prod_j S^3_{\sqrt{2 r_j}}$ via the inclusion map $\imath$ as
in \eqref{3sphere}.
Recall that the Hopf map $H^n$ maps $(a,b)$ to $(e_1, \ldots, e_n) \in \prod_j
S^2_{r_j}$ where $e_i= 1/2 (|a_i|^2 - |b_i|^2, 2 \textnormal{Re} (\bar{a}_i
b_i), 2 \textnormal{Im} (\bar{a}_i b_i)).$

Condition \eqref{zk0} then implies that for every $k= p+1, \ldots, n-1$ the
vectors $e_k $ are positive multiples of each other, i.e. 
\begin{equation}\label{eq:paralleli}
\forall k = p+1, \ldots, n-1, \, \exists \lambda_k \in \R_+ \, \textnormal{ s.t.
} e_k = \lambda_k e_n.
\end{equation}
In fact, if $a_n=0$ then $a_k=0$  (note that $b_n$ can not vanish simultaneously
with $a_n$ since we assumed $r_n \neq 0$) and clearly
$$ e_k = \Big(- \frac{|b_k|^2}{2}, 0, 0 \Big) = \frac{|b_k|^2}{2 |b_n|^2} \Big(-
|b_n|^2, 0, 0 \Big)= \frac{|b_k|^2}{2 |b_n|^2} e_n.$$ 
Similarly if $b_n=0$ then 
$$ e_k = \Big(\frac{|a_k|^2}{2}, 0, 0 \Big) = \frac{|a_k|^2}{2|a_n|^2} e_n.$$
If both $a_n$ and $b_n$ are non-zero, from  \eqref{zk0} we obtain 
$ b_k \bar{a}_k = \frac{b_n}{a_n} |a_k|^2$. This implies 
$$ |b_k|^2 = \frac{\bar{b}_n}{\bar{a}_n} \bar{a}_k b_k = 
\frac{|\bar{b}_n|^2}{|\bar{a}_n|^2} |a_k|^2.$$
It then follows that $$ e_k = |a_k|^2 \Big( \frac{1}{2} \big(1 -
\frac{|\bar{b}_n|^2}{|\bar{a}_n|^2} \big), Re \frac{b_n}{a_n}, Im\frac{b_n}{a_n}
\Big) = \frac{|a_k|^2 }{2 |\bar{a}_n|^2} e_n$$
proving \eqref{eq:paralleli}.

Therefore $$ H^n(\imath \{(a,b) \in  p^{-1}(\mu_{U(1)^n}^{-1} (r)) \mid  a_kb_n
- a_n b_k=0 \quad \forall \, k= p+1, \ldots, n-1 \}) = \widetilde{M}_{I_p^c} (r)
$$
and the projective space \eqref{eq:CP2} is then the quotient $ M_{I_p^c}(r):=
\widetilde{M}_{I_p^c} (r)/ SO(3)$.
In words, $M_{I_p^c}(r)$ is the submanifold of $M(r)$ of those $n$-gons such
that the last $n-p$ edges are positive multiple one of the other.
Note that for $r \in \Delta^0$ this is a non-empty submanifold (in fact it is a
$\C\P^{p-2}$). In particular, $M_{I_p^c}(r)$  is naturally diffeomorphic to the
$(p-2)$-dimensional polygon space $M(r_1, \ldots, r_p, \sum_{k=p+1}^n r_k).$

On the other hand, for $$-\sum_{j=2}^p |w_j|^2 + \sum_{k=p+1}^{n-1} |z_k|^2 = 
\epsilon$$ the orbit space $\mu_{S^1}^{-1}( \epsilon) / S^1 $ is identical
topologically with $W_+$. Again $r_1$ and $r_n $ are uniquely determined:
$$ r_1 = 1- \sum_{j=2}^p r_j^c - \epsilon \quad \textnormal{and} \quad 
r_n =  1- \sum_{k=p+1}^{n-1} r_k^c + \epsilon.$$
The resulting lengths vector $r^1 = (r_1, r_2^c, \ldots, r_{n-1}^c, r_n)$
satisfies $- \epsilon_{I_p}(r^1) = \epsilon_{I_p^c}(r^1) >0$.
The zero section of $W_+$ is now the projective space $\C\P^{q-2}$ corresponding
to the vanishing of the coordinates $w_j$ for $j=2, \ldots, p$. It follows then
from \eqref{zw} that $w_j$ vanishes if and only if $a_1 b_j - a_j b_1 =0$.
Arguments similar to the ones above allow us to identify $\C\P^{q-2}$ with the
submanifold $M_{I_p} (r)\subset M(r)$. 
Again, for $r \in \Delta^1$ the submanifold $M_{I_p}(r)$ is non empty and is
diffeomorphic to the $(q-2)$-dimensional polygon space $M( \sum_{j=1}^p r_j,
r_2, \ldots, r_n).$

Note that on the wall crossing point $r^c,$ $\eps_{I_p}(r^c)=0$ and
$M_{I_p}(r^c)=M_{I_p^c}(r^c)$ is the singular point $[P^c]$ in $M(r^c).$
Moreover, note that as $r^0 \rightarrow r^c,$ we have $\epsilon_{I_p}(r^0)
\rightarrow 0$ and $ M_{I_p^c}(r^0)\subset M(r^0)$ collapses to $[P^c] \in
M(r^c).$ Similarly, as $r^1 \rightarrow r^c,$ we have $\epsilon_{I_p}(r^1)
\rightarrow 0$ and $ M_{I_p}(r^1)\subset M(r^1)$ collapses to $[P^c] \in
M(r^c).$
Roughly speaking, as $r^0 \rightarrow r^c,$  the ``width'' $\epsilon_{I_p}(r^0)$
of polygons in  $ M_{I_p^c}(r^0)\subset M(r^0)$  goes to zero, and the
$(p-2)$-dimensional submanifold  $M_{I_p^c}(r)$ collapses to a point when $r$
reaches the wall $W_{I_p}.$ 
Similarly, as $r$ leaves from the wall $W_{I_p}$ to the interior of $\Delta^{\!
1}$, the degenerate polygon $[P^c]$ gets inflated of an  $\epsilon_{I_p^c}(r^1)$
amount, and $M_{I_p}(r^1)$ is the $(q-2)$-dimensional submanifold that is born
as crossing the wall $W_{I_p}$.

The birational map between $M(r^0)$ and $M(r^1)$ is hence the composite of 
a blow up followed by a blow down, where the exceptional divisor is the
product of the flip loci. The maps $\beta_+$ and $\beta_-$ blow down the flip 
loci $M_{I_p^c}(r)$ and $M_{I_p}(r)$ to the singular point $[P^c]\in M(r^c)$, 
as in Figure \ref{scop}. Note that in Figure \ref{scop} there are no moment
polytopes, just schematic representations of the (eventually singular)
manifolds.

\begin{figure}[htbp]
\begin{center}
\psfrag{p}{\footnotesize{$\mathbb{C} \mathbb{P}^{p-2}$}}
\psfrag{q}{\footnotesize{$\mathbb{C} \mathbb{P}^{q-2}$}}
\psfrag{M_1}{\footnotesize{$M(r^1)$}}
\psfrag{M_0}{\footnotesize{$M(r^0)$}}
\psfrag{M_c}{\footnotesize{$M(r^c)$}}
\psfrag{M~}{\footnotesize{$\tilde{M}$}}
\psfrag{p_1}{\footnotesize{$p_0$}}
\psfrag{p_2}{\footnotesize{$p_1$}}
\psfrag{P}{\footnotesize{$P^c$}}
\psfrag{pi0}{\footnotesize{$\beta_-$}}
\psfrag{pi1}{\footnotesize{$\beta_+$}}
\psfrag{wall}{\footnotesize{wall}}
\psfrag{crossing}{\footnotesize{crossing}}
\psfrag{-sum}{\tiny{$-\sum r_i$}}
\psfrag{sum}{\tiny{$\sum r_i$}}
\includegraphics[width=9cm]{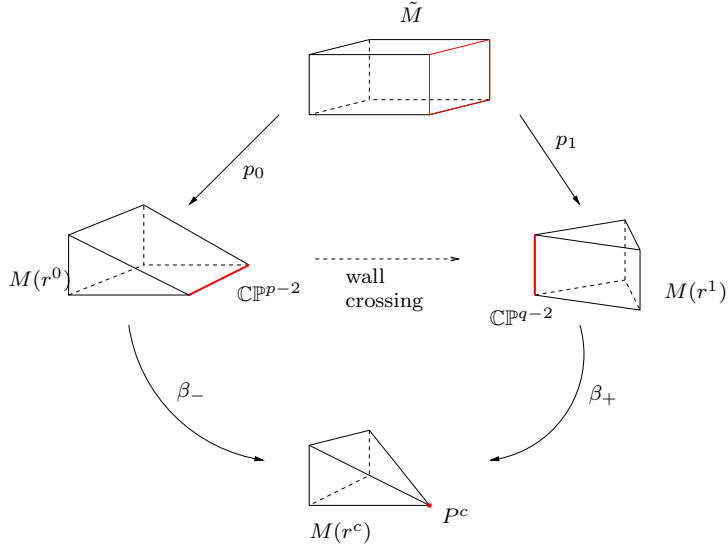}
\caption{Crossing a wall.}
\label{scop}
\end{center}
\end{figure}

\end{proof}

Note that the above wall-crossing analysis also holds for external walls. So in
particular, for any $ i=1, \ldots, n$, crossing the wall $W_{\{1, \ldots, n\}
\setminus \{i\}}$ replaces the empty set with $M(r) \simeq M_{\{i\}} (r) \simeq
\C\P^{n-3}$. Therefore we have the following immediate consequence of Theorem
\ref{WCtheo}
\begin{prop}\label{outer walls}
Let $r$ be in an external chamber $\Delta$ of $\Xi$. Then  $$M(r) \simeq
\C\P^{n-3}.$$
\end{prop}
In particular, Proposition \ref{outer walls} implies that 
for $n\geq5$ the Weyl group acts transitively on the external chambers of $\Xi$.
In fact, for $n\geq 5$, Farber, Hausmann and Sch\"utz \cite{fhs} have shown
that,
for $r$'s in different chambers of $\Xi$, the polygon spaces $M(r)$ are not 
diffeomorphic unless their chambers are related by the Weyl group action.

\subsection{Examples}\label{es wall-crossing}
Let $\Delta^{\!0} $ and $\Delta^{\!1}$ be the adjacent chambers in $\Xi\in \R^5$
as in
Example \ref{esvolume}.
Note that $\{3\}$ is a maximal short set in $\mathcal S(\Delta^{\!0})$ and
therefore $\Delta^{\!0}$ is an external chamber. By Corollary \ref{outer walls},
$M(r^0) \simeq \C\P^{2}$.

The closures of $\Delta^{\!0}$ and $\Delta^{\!1}$ intersect in the wall of
equation  $$\epsilon_{\{1,3\}} (r)= \epsilon_{\{ 2,4,5\}} (r) =0,$$ see Example
\ref{esvolume}. 
In particular, if one considers the lengths vectors $r^0 = \frac{1}{7}
\big( 1,1,3,1,1 \big) \in \Delta^{\!0}$ and $r^1= \frac{2}{11} \big(
\frac{1}{2},1,2,1,1 \big) \in \Delta^{\!1}$, then the segment $[r^0, r^1]$
 hits the wall at $r^c = \frac{1}{6} \big( \frac{2}{3},1,\frac{7}{3},1,1 \big)$.
By Theorem
\ref{WCtheo}, when crossing this wall the point $ M_{\{ 2,4,5\}} (r^0) \simeq
M(r_2^0+r_4^0+r_5^0, r_1^0, r_3^0)$ gets blown up. Therefore, for all $r^1 \in
\Delta^{\! 1},$ $M(r^1) $ is
diffeomorphic to $\C\P^2 $ blown up in one point with exceptional divisor
$ M_{\{1,3\}} (r^1) \simeq M(r^1_1 +r^1_3, r^1_2, r^1_4, r^1_5) \simeq \C\P^1$.
The maps $\beta_+$ and $ \beta_-$ blow down $ M_{\{ 2,4,5\}} (r^0)$ and $
M_{\{1,3\}} (r^1) $ to the critical point $[P^c]$ where the degenerate polygon
has edges
$$e_1^c = \Big(\frac{1}{9}, 0, 0\Big), \quad e_2^c = \Big(-\frac{1}{6}, 0, 0
\Big)= e_4^c =
e_5^c, \quad e_3^c = \Big(\frac{7}{18}, 0, 0 \Big).$$

\section{The Cohomology Ring of $M(r)$}\label{coho}

In this Section we study how the cohomology ring of $M(r)$ changes as $r$
crosses a wall in the moment polytope $\Xi.$ We apply the Duistermaat--Heckman
Theorem together with the volume formula (Theorem \ref{volume}) to describe
explicitly the cohomology ring $H^*(M(r)).$

The study of the cohomology ring structure of a reduced space $M/ \!\!/G$ has
been since the 80's one of the foremost  topics in equivariant symplectic
geometry. 
The problem is not closed though. In fact (even in the well-behaved case of a
compact connected Lie group $G$ acting on a compact manifold $M$) to give an
explicit description of the cohomology ring $H^*(M/ \!\!/G, \C)$ still some (non
trivial) work needs to be done. This was already pointed out by Guillemin and
Sternberg in \cite{gs95}, who observed that in ``nice'' situations (essentially
when the Chern class of the fibration $\mu^{-1}(\xi) \rightarrow M/ \!\!/G $
generates the cohomology ring), then a good deal of information on $H^*(M/
\!\!/G, \C)$ can be deduced from the Duistermaat--Heckman Theorem, if the
polynomial that describes the volume of the symplectic reduction is known. This
is the point of view we take in our analysis. 

\subsection{The cohomology ring of reduced spaces}\label{The cohomology ring of
reduced spaces}

In this Section we summarize the main ideas and theorems in \cite{gs95} using
the notation of moduli spaces of polygons. These arguments are valid in more
general settings, and  have been applied in \cite{gs95} to flag manifolds and
toric manifolds associated with a simplicial fan. For proofs and more details we
refer to \cite{gs95}. 
Let $r$ and $r^0$ be regular values of $\mu_{U(1)^n}$ lying in the same chamber
and denote by $(M(r), \omega_r)$ and $(M(r^0), \omega_{r^0})$
the associated symplectic quotients. Using this notation we now state the
Duistermaat-Heckman Theorem,  which relates the cohomology classes $[\omega_r]$
and $[\omega_{r^0}]$ of the symplectic reduced forms $\omega_r$ and
$\omega_{r^0}.$

\begin{theorem} \textnormal{(J.J.~Duistermaat, G.J.~Heckman, \cite{DH})}
\label{DH}As differentiable manifolds $M(r) = M(r^0),$ and 
$$[\omega_{r}]= [\omega_{r^0}]+ \sum_{i=1}^n (r_i-r_i^0) c_i $$
where $c_1, \ldots, c_n$ are the first Chern classes 
of the $n$ line bundles associated to the 
fibration  $\mu^{-1}(r) \rightarrow M(r).$
\end{theorem}

By definition of symplectic volume, we have:
$$\vol \,M(r)= \int_{M(r)} \exp([\omega_r]) = \int_{M(r^0)} \exp([\omega_{r^0}]
+  \sum_{i=1}^n (r_i-r_i^0) c_i).$$ Then $\vol \,M(r)$ is a polynomial (on each
chamber) of degree $n-3$ and

$$\frac{\d^{\a}}{\d r^{\a}} \vol \, M(r)_{|_{r^0}}= \frac{1}{k!} \int_{M(r^0)}
[\omega_{r^0}]^k  c_1^{\a_1} \cdots c_n^{\a_n} $$
for $\a$ multindex, $|\a| = \a_1+ \ldots +\a_n = n-3-k,$ with $0\leq k \leq
n-3.$

In particular, if $|\a|=n-3,$
\begin{equation}\label{mah}
\frac{\d^{\a}}{\d r^{\a}} \vol \, M(r)_{|_{r^0}}=  \int_{M(r^0)} c_1^{\a_1}
\cdots c_n^{\a_n}.
\end{equation}

If the $c_i$ generate the cohomology ring $H^*(M(r), \C),$ then Guillemin and
Sternberg observe that it is possible to read from \eqref{mah} the
multiplicative relations between the generators, concluding the following
explicit description of $H^*(M(r), \C)$.
\begin{theorem} \label{123} (\cite{gs95})
If $c_1, \ldots, c_n$ generate the cohomology ring  $H^*(M(r), \C),$ then 
$H^*(M(r), \C)$ is isomorphic to the abstract ring $$ \C[x_1, \ldots, x_n]/
\ann(\vol \, M(r))$$ where $Q(x_1, \ldots, x_n) \in \ann(\vol \, M(r))$ if and
only if $Q(\partial/\partial {r_1},\ldots ,\partial/\partial {r_n})\vol \,
M(r)=0$ and the isomorphism is given by $x_i \mapsto c_i$.
\end{theorem}


Therefore it is a central problem to determine when the $c_i$ generate the
cohomology ring $H^*(M(r)).$ When $M(r)$ is toric it is well known that this is
the case (see for example \cite{fulton}). Still, there are choices of $r$ for
which the polygon space $M(r)$ is not
toric, as it is the case, for example, for $r=(1,1,1,1,1)$, cf. \cite{hk2} and
Remark \ref{bending}.

In general let $\Delta$ be the set of regular values of $\mu$ in the convex
polytope $\Xi.$
The connected components $\Delta^{\! 1}, \ldots, \Delta^{\ell}$ of $\Delta$ are
themselves convex polytopes. Therefore, by the Duistermaat-Heckman Theorem, the
diffeotype of the reduced space $M(r)$ (thus also its cohomology ring) depends
only on the chamber $\Delta^{\! i}$ that contains $r.$
If the closure of $\Delta^{\! i}$ contains a vertex of $\Xi,$ then its
associated reduced space is a toric manifold and its associated cohomology ring
is generated by  the $c_i,$ \cite{gs95}. 

We prove that this holds for each regular value $r$ by applying the
wall-crossing analysis we did in Section \ref{crossing the walls}, showing that
crossing a wall has the effect of killing some relations, and so (roughly
speaking) some of the generators that were ``hidden'' appear.

\subsection{Wall-crossing and Cohomology}\label{wallcrossing and cohomology}

By Theorem \ref{WCtheo}, when $r$ crosses the wall $W_{I_p}$ the diffeotype of
the reduced manifold $M(r)$ changes by  replacing a copy of $\C\P^{p-2}$ in
$M(r)$ by a $\C\P^{q-2}$ by means of a blow-up followed by a blow-down.

In this Section we study how the cohomology ring $H^*(M(r))$ changes as $r$
crosses a wall; the main tools to prove our result are the Mayer--Vietoris
sequence and the Gysin sequence, together with the decomposition Theorem as
presented in \cite{bbd} and \cite{luca}.

Let $M$ and $M'$ be the moduli spaces of polygons respectively before and after
crossing the wall $W_{I_p}$. Moreover, denote by $V$ and $V'$ the tubular
neighborhoods in $M$ and $M'$ respectively of the submanifolds as in Theorem
\ref{WCtheo}: 
\begin{displaymath}
\begin{array}{c}
 V=N_{\varepsilon} \C\P^{p-2}= \textnormal{tubular neightborhood of}
\,\,\C\P^{p-2} \subset M\\
V'=N_{\varepsilon} \C\P^{q-2}=  \textnormal{tubular neightborhood of}
\,\,\C\P^{q-2} \subset M'\\
U= M \setminus \C\P^{p-2}\\
U'= M' \setminus \C\P^{q-2}.
\end{array}
\end{displaymath}
By the wall-crossing Theorem \ref{WCtheo}, $U=U'$ and $U\cap V = U'\cap V'=:
S_{\varepsilon}.$ The Mayer--Vietoris sequences for the manifolds $M$ and $M'$
are:
$$\ldots \rightarrow H^{k-1}(S_{\e}) \rightarrow H^k(M) \rightarrow H^k(U)
\oplus H^k(V) \rightarrow H^k(S_{\e}) \rightarrow \ldots $$
$$\ldots \rightarrow H^{k-1}(S_{\e}) \rightarrow H^k(M') \rightarrow H^k(U')
\oplus H^k(V') \rightarrow H^k(S_{\e}) \rightarrow \ldots $$

Because  $H^k(U)=H^k(U'),$ the change in the cohomology ring structures $H^*(M)$
and $H^*(M')$ is enclosed in how $H^k(V')$ and $H^k(V)$ map into $H^k(S_{\e}).$
These maps will be brought to light in the proof of the next proposition.
\begin{prop}
$$H^*(S_{\e})= H^*(\C\P^{\textnormal{min}(p,q)-2}) \otimes
H^*(S^{2\textnormal{max}(p,q)-3 })$$
\end{prop}

\begin{proof}
By construction, $N_{\e}\C\P^{p-2}$ is the total space of a disk fibration over
$\C\P^{p-2},$ and $S_{\e}$ is the total space of the associated fibration in
spheres:

$$\xymatrix{ N_{\e}\C\P^{p-2} \ar[d]^{D^{2q-2} \quad \Rightarrow}  & S_{\e}
\ar[d]^{S^{2q-3}}\\
\C\P^{p-2} & \C\P^{p-2}\\}$$
The sphere fibration 
$ \xymatrix@1{\pi:S_{\e} \ar[r]^-{S^{2q-3}} &\C\P^{p-2 }}$ 
induces the following Gysin sequence
$$  \xymatrix{{} \ar[r]& H^k(\C\P^{p-2}) \ar[r]^-{\pi^*} &H^k(S_{\e})
\ar[r]^-{\pi_*} &  H^{k-(2q-3)}(\C\P^{p-2}) \ar[r]^-{\wedge e}& 
H^{k+1}(\C\P^{p-2}) \ar[r] & {}    } $$
where $\pi^*$ is the map induced in cohomology by the projection map $\pi,$
$\pi_*$ is the integration along the fibers and $\wedge e$ is the wedge product
with the Euler class.

Recall that
$$H^k(\C\P^{p-2})= \left\{ 
\begin{array}{ll}
\C & \textnormal{if} \,\,k= 0,2, \ldots, 2(p-2)\\
0& \textnormal{otherwise}.
\end{array}\right. $$
Suppose that $q \geq p.$ Then the first part of the Gysin map is
$$\xymatrix{ \C \ar[r]^-{\pi^*} & H^0(S_{\e}) \ar[r]^-{\pi_*} & 0
\ar[r]^-{\wedge e}& 0 \ar[r]^-{\pi^*} & H^1(S_{\e}) \ar[r]^-{\pi_*} & 0
\ar[r]^-{\wedge e}& \C \ar[r]^-{\pi^*} & {}\\
{} \ar[r]^-{\pi^*} & H^2(S_{\e}) \ar[r]^-{\pi_*} & 0 \ar[r]^-{\wedge e} & 0
\ar[r]^-{\pi^*} & H^3(S_{\e})\ar[r]^-{\pi_*} & 0 \ar[r]& \ldots & {}}$$
until cohomology groups of degree $k=2p-2$ (in fact  $H^{k-(2q-3)}(\C\P^{p-2})
\simeq 0 $ for all  $0 \leq k \leq 2q-3$). Therefore
\begin{equation}\label{22-1}
H^k(S_{\e}) \simeq H^{k} (\C\P^{p-2}) \quad \forall \,\,0 \leq k \leq 2(p-2). 
\end{equation}

At $k= 2q-3$ the Gysin sequence goes as follows:
$$\xymatrix{ 0 \ar[r]^-{\pi^*} & H^{2q-3} (S_{\e}) \ar[r]^-{\pi_*} & \C
\ar[r]^-{\wedge e} & 0 \ar[r]^-{\pi^*}& H^{2q-3} (S_{\e}) \ar[r]^-{\pi_*} & 0
\ar[r]^-{\wedge e} & {}\\
& {} \ar[r]^-{\wedge e} & 0  \ar[r]^-{\pi^*}& H^{2q-1} (S_{\e}) \ar[r]^-{\pi_*}
& \C \ar[r]^-{\wedge e}&0 \ar[r]& \ldots}$$
To check this second part the only thing to keep in mind is that
$H^{k}(\C\P^{p-2}) \simeq 0$ for all $k \, \geq 2q-3. $ In fact $k \geq 2q-3
\geq 2p-3 > 2(p-2)$. Observing that $k-2q+3=2(p-2) \iff k = 2(p+q)-7,$ 
\begin{equation}\label{22-2}H^k(S_{\e}) \simeq H^{k-(2q-3)} (\C\P^{p-2}) \quad
\forall \,\,2q-3 \leq k \leq 2n-7. \end{equation}
\begin{equation}\label{22-3}H^k(S_{\e}) \simeq 0 \quad \forall \, k : \, 2(p-2)
< k <2q-3 \quad \text{or} \quad k \geq 2n-6. \end{equation}
Summarizing, from \eqref{22-1}, \eqref{22-2} and \eqref{22-3} we get
$$H^*(S_{\e})= H^*(\C\P^{p-2}) \otimes H^*(S^{2q-3 }).$$
It is easy to check that if we assume $p \geq q$ then $p$ and $q$ exchange their
roles, and the result follows.
\end{proof}

Note that $H^*(V)=H^*(\C\P^{p-2})$ because $V$ retracts on $\C\P^{p-2}.$
Similarly, $H^*(V')=H^*(\C\P^{q-2}),$ and we have all the ingredients to write
the Mayer--Vietoris sequences for $M$ and $M':$  
$$ H^0(M) \rightarrow H^0(U) \oplus \C \rightarrow \C \rightarrow H^1(M)
\rightarrow H^1(U) \oplus 0 \rightarrow 0 \rightarrow$$ $$\rightarrow H^2(M)
\rightarrow H^2(U)\oplus \C \rightarrow \C \rightarrow$$

and
$$ H^0(M') \rightarrow H^0(U') \oplus \C \rightarrow \C \rightarrow H^1(M')
\rightarrow H^1(U') \oplus 0 \rightarrow 0 \rightarrow$$ $$\rightarrow H^2(M')
\rightarrow H^2(U')\oplus \C \rightarrow \C \rightarrow$$
Assume again $q \geq p.$ So, until degree $2(p-2),$ the two sequences above are
the same, thus $$H^k(M)= H^k(M') \quad \forall \,0 \leq k \leq 2(p-2).$$
At $2(p-2)+1$ the Mayer--Vietoris sequences of the manifolds of $M$ and $M'$
are:
$$ \rightarrow H^{2p-3}(M) \rightarrow H^{2p-3}(U) \oplus 0 \rightarrow 0
\rightarrow H^{2p-2}(M) \rightarrow H^{2p-2}(U) \oplus 0 \rightarrow 0
\rightarrow$$
$$ \rightarrow H^{2p-3}(M') \rightarrow H^{2p-3}(U') \oplus 0 \rightarrow 0
\rightarrow H^{2p-2}(M') \rightarrow H^{2p-2}(U') \oplus \C \rightarrow 0
\rightarrow$$
and, until degree $2q-3$ the two sequences differ by the fact that $H^{k}(V')=
\C$ while $H^k(V) \simeq 0$ for $k$ even, $2p-3 \leq k \leq 2q-3$. 
Thus $$ \textnormal{dim}(H^k(M'))= \textnormal{dim}(H^k(M)) +1 \quad
\textnormal{if} \, k \, \textnormal{even}, \, 2p-3 \leq k \leq 2q-3  $$
and
$$ H^k(M')= H^k(M)=0 \quad \textnormal{for} \, k \, \textnormal{odd}.$$
At $2q-3$ the Mayer--Vietoris sequences for $M$ and $M'$ are 
$$ \rightarrow H^{2q-3}(M) \rightarrow H^{2q-3}(U) \oplus 0 \rightarrow 0
\rightarrow H^{2q-2}(M) \rightarrow H^{2q-2}(U) \oplus 0 \rightarrow 0
\rightarrow$$
$$ \rightarrow H^{2q-3}(M') \rightarrow H^{2q-3}(U') \oplus 0 \rightarrow 0
\rightarrow H^{2q-2}(M') \rightarrow H^{2q-2}(U') \oplus 0 \rightarrow 0
\rightarrow$$
and so again (just as for $0 \leq k \leq 2(p-2)$) $$H^k(M') \simeq H^k(M) \quad
k \geq 2q-3. $$

If $p \geq q $ similar arguments hold, therefore we have proved the following:
\begin{prop}
$$ \begin{array}{ll} H^k(M')= H^k(M) =0 & \textnormal{if} \, \,k \,\,
\textnormal{is odd};\\
 H^k(M')= H^k(M) & \left\{ \begin{array}{l}
0 \leq k \leq 2(\textnormal{min}(p,q)-2),\\
k \geq 2 \textnormal{max}(p,q)-3;
\end{array} \right.\\
\textnormal{dim} H^k(M')= \textnormal{dim} H^k(M) +1 &  k \, \textnormal{even},
2p-2 \leq k \leq 2q-4 \,(\textnormal{case} \, \,q \geq p );\\
\textnormal{dim} H^k(M)= \textnormal{dim} H^k(M') +1 &  k \, \textnormal{even},
2q-2 \leq k \leq 2p-4 \,(\textnormal{case} \, \,p \geq q );\\
\end{array}$$
\end{prop}

This calculation, done using the Mayer--Vietoris sequences of the ma\-nifolds
$M$ and $M',$ tells us in which degree the cohomology groups of the symplectic
quotient $M(r)$ change as $r$ crosses a wall $W_{I_p}$.

Even though it is natural- by the construction- to expect that the new born
cohomological classes are polynomial in the class of the blown up manifold
$\C\P^{q-2}$ or $\C\P^{p-2}$, this calculation does not give us such precise
informations. We use the decomposition Theorem due to
Beilinson--Bernstein--Deligne \cite{bbd} to identify precisely the new born
classes that increase the dimension of the cohomology groups of ``middle''
degrees. To this aim, some notation needs to be introduced.

Let $f: X \rightarrow Y$ be a map of algebraic manifolds (i.e. manifolds which
are the set of common zeros of a finite number of polynomials). For each $0<k<
\frac{\textrm{dim} X}{2},$ define  
$Y_k:= \{ y \in Y : \textnormal{dim}(f^{-1}(y))\geq k \}.$
The map $f$ is \emph{small} if and only if
\begin{equation}\label{21} \textnormal{dim}Y_k +2k < \textnormal{dim}X \quad
\forall \,\, 0<k< \frac{\textrm{dim} X}{2}; \end{equation}
and \emph{semi-small} if and only if 
\begin{equation}\label{22} \textnormal{dim}Y_k +2k \leq \textnormal{dim}X \quad
\forall \,\, 0<k< \frac{\textrm{dim} X}{2}. \end{equation}

\begin{prop}
At least one of the blow down maps $\beta_+$ and $\beta_-$ as in Theorem
\ref{WCtheo} is small.
\end{prop}
\begin{proof} 
Denote by $ Y_k^{\pm}:= \{ y \in M(r^c) :
\textnormal{dim}(\beta_{\pm}^{-1}(y))\geq k \}.$
If $q > p,$ then $\beta_+$ is small. In fact 
\begin{displaymath}
Y_k^+= \left\{\begin{array}{ll}
\{[P^c]\} & \textnormal{for } 1 \leq k \leq 2(p-2)\\
\emptyset & \textnormal{otherwise}.
\end{array} \right.
\end{displaymath}
Thus inequality \eqref{21} for $k= 2(p-2)$ is verified: $$ 4(p-2) < 2(n-3) \iff
4p -8 < 2p+2q-6 \iff p-1 < q,$$ and similarly inequality \eqref{21} is verified
for smaller $k$'s.
Under the assumption $q >p$ the map $\beta_-$ is not semi-small (thus even not
small). In fact $Y_k^-= \{[P^c] \}$ and  inequality \eqref{22} for $k=q-2$ does
not hold since $$4(q-2) \leq 2(n-3 ) \iff q-1 \leq p.$$ 

If $ p > q,$ then $\beta_-$ is small and $\beta_+$ is not semi-small. 
Note that if $p=q$ then both $\beta_-$ and $\beta_+$ are small.
\end{proof}

Assume that $\beta_+: M \rightarrow M(r^c)$ is small. Then $ H^*(M) = I
H^*(M(r^c)),$ where  $IH^*(M(r^c))$ is the intersection cohomology  of the
singular manifold $M(r^c), $ (see the survey paper by M.~de Cataldo and
L.~Migliorini  \cite{luca2}).

We state the decomposition Theorem just for the special situation of $\beta_+$
and $\beta_-$ resolutions of the singularity corresponding to the lined polygon
in $M(r^c).$
For the statement in full generality, proofs and more details we refer to the
original paper \cite{bbd}, and to \cite{luca} by  de~Cataldo--Migliorini, where
an alternative proof is given.

In our setting, the decomposition Theorem says that $H^*(M')$  is isomorphic to
the intersection cohomology  $IH^*(M(r^c))$ of $M(r^c)$ plus polynomials in the
cohomological classes of submanifolds $\mathcal C_i$ of $M.$ In the moduli space
si\-tuation, these submanifolds are just the preimages of the points $y_i \in
Y_k^+$.

If we assume $q \geq p$ (which is equivalent to assuming $\beta_+$ small), then
$ C:= (\beta_-)^{-1} ([P^c])$ is the resolution in $M'$ of the singularity
$[P^c].$ By Theorem \ref{WCtheo}, $$\mathcal C = M_{I_p}(r) \simeq \C\P^{q-2}.$$

Applying the decomposition Theorem we get:
\begin{theorem}\label{decomposition}
Let $\beta_+ : M \rightarrow M(r^c)$ be a small resolution of the singularity in
$ M(r^c)$ and let $M'$ be the polygon space birational to $M$ via the single
wall-crossing described above. Then
$$H^*(M') = H^*(M)\oplus \bigoplus_{\a=0}^{q-p} \C \Big( PD([M_{I_p}(r)]) \smile
c_1^{\a} (\mathcal{N'}) \Big)$$
where  $PD([M_{I_p}(r)]) \in H^{2p-2}(M')$ is the Poincar\'e dual of $ 
M_{I_p}(r) \subset M',$ and   $c_1(\mathcal{N'}) $ is the first Chern class of
the normal bundle  $\mathcal{N'} $ to $ M_{I_p}(r) \subseteq M'.$
\end{theorem}

At the light of this result, to prove that $H^*(M(r))$ is generated by the Chern
classes $c_i$ we need to express the classes $PD([M_{I_p}(r)])$ and the cup
products $PD([M_{I_p}(r)]) \smile c_1^{\a} (\mathcal{N'} )$ as combinations of
the $c_i.$

By Poincar\'e duality, $PD([M_{I_p}(r)]) \in H^{2p-2}(M')$ and thus we want to
show that, for some constants $A_{\a},$
$$ PD([M_{I_p}(r)])= \sum_{\sum \a_i= p-1} A_{\a} c_1^{\a_1} \cdots
c_n^{\a_n}.$$

To this aim, we will explicitly describe the classes $c_i$ by means of the two
description of the polygon space $M(r)$ that one gets from the $U(1)^n \times
U(2)$-action on $\C^{n \times 2}$ by performing reduction in stages as
summarized in diagram \ref{beauty}.

Let $r$ be a regular value in $\Delta^{\! 1}$ such that the reduced manifold
$M(r)\simeq M'$. 
Since the fibration $\mu^{-1}(r) \rightarrow M(r)$ as in \eqref{beauty} is
trivial,  the classes $c_i$ of the $n$-complex line bundles associated to it 
are actually the classes  $c_i$
relative to the $n$ Hopf fibrations $ S^3_{\sqrt{2 r_i}}  \rightarrow 
S^2_{r_i}.$  These are well known to be the Chern classes of the tautological
line bundle $\mathcal O(-1)$ over $\C\P^1$ (under the identification $\C \P^1
\simeq S^2$).
More precisely
\begin{equation}\label{-}
 s^* c_i = c_1 (p_i^*(\mathcal O(-1))) = - [p_i^* \omega_{FS}]
\end{equation} where $s$ is
the fibration $s: \mu_{SO(3)}^{-1}(0) \rightarrow M(r)$, $p_i$ is the projection
$p_i: \prod_j S^2_{r_j} \rightarrow S^2_{r_i}$ and $\omega_{FS}$ is the
Fubini--Study symplectic form.

\begin{prop}\label{MIP}
The Poincar\'e dual of $M_{I_p} (r)$ is the $(2p-2)-$class 
$$PD ([M_{I_p} (r)])= (-1)^{p-1} \prod_{j=2, \ldots, p} (c_{i_j} + c_{i_1}) \in 
H^{2p-2}(M')$$
where $I_p = \{i_j \mid j=1, \ldots, p \}$. Moreover  first Chern class
$c_1(\mathcal N')$ of the normal bundle to $M_{I_p}(r)$ is
$$ c_1 (\mathcal N' ) =-2( c_{i_2} + \cdots + c_{i_p}).$$
\end{prop}
\begin{proof}
The polygon space $M(r)$ can also be described as the GIT quotient of $\prod_n
\C\P^1$ by the diagonal action of $PSL(2, \C).$ Let $H \rightarrow \C\P^1$ be
the hyperplane bundle. Then the line bundle $p_i^* H \otimes p_j^* H $ over
$\prod_n \C\P^1$ induces a line bundle $L_{ij}$ on the quotient $M(r)$, cf
\cite{konno}. 
For all $i=1, \ldots,n $, let $z_i$ be the first Chern class
$$ z_i := c_1 (L_{ii}) \in H^2(M(r), \mathbb Z).$$ 
Clearly $s^* z_i = c_1 (p_i^* \mathcal O(2)) = 2 [p_i^* \omega_{FS}].$ It
follows that 
$$ z_i = -2c_i.$$
In the equilateral case Kamiyama and Tezuka \cite{kamiyama} prove that the
Poincar\'e dual of $\frac{z_i + z_j}{2}$ is the submanifold of $M(1, \ldots, 1)$
consisting of those polygons $\vec{e}$ such that $ e_i=e_j.$ This easily
generalizes to the non-equilateral case (cf \cite{konno}), and the Poincar\'e
dual of $\frac{z_i + z_j}{2} = - (c_i +c_j)$ is the submanifold $M_{\{i,j \}}
(r)$.
Since $$ M_{I_p} (r)= \bigcap_{i_j \in I_p \setminus \{ i_1\}} M_{\{i_1,i_j \}}
(r) $$ the result follows.

Analogously one can prove that the first Chern class $c_1(\mathcal N')$ of the
normal bundle to $M_{I_p}(r)$ is a linear combination of the Chern classes $c_i,
i= 1, \ldots, n.$ In fact the tangent bundle to $M(r)$ is the direct sum  $p_1^*
\mathcal O(2) \oplus \cdots \oplus p_n^* \mathcal O(2)$ of the pullbacks of the
tangent bundle to each sphere. 

The submanifold $M_{I_p}(r)$ is the moduli space of polygons obtained as the
quotient by the $SO(3)$-action on the product of spheres of radii $(\sum_{i \in
I_p} r_i,r_{i_{p+1}}, \ldots, r_{i_n}),$ with $I_p=\{i_1, \ldots, i_p\}$. Since
the inclusion $S^2_{\sum_{i \in I_p} r_i} \hookrightarrow S^2_{r_{i_1}} \times
\ldots \times S^2_{r_{i_p}}$ is the diagonal one, then the tangent bundle to
$M_{I_p}(r)$ is  $$(p_{i_1}^* \mathcal O(2)\otimes \cdots \otimes 
p_{i_p}^*\mathcal O(2)) \oplus p_{i_{p+1}}^* \mathcal O(2)\oplus \cdots \oplus
p_{i_n}^* \mathcal O(2)$$
and therefore the first Chern class of the quotient $T M(r) / T M_{I_p}(r)$ is
the sum of $(p-1)$ among the first Chern classes $c_i$ for $i \in I_p,$ i.e. 
\begin{equation}\label{eq:normal}
c_1 (\mathcal N' ) =-2( c_{i_2} + \cdots + c_{i_p}). 
\end{equation}
\end{proof}

Note that similar arguments hold if $p >q$, i.e. if $\beta_+$ is not small while
$\beta_-$ is. In this case the decomposition Theorem implies that the cohomology
of $M$ is described as follows
$$ H^*(M) = H^*(M')\oplus  \bigoplus_{\a=0}^{p-q} \C \Big( PD([ M_{I_q}(r)])
\smile c_1^{\a} (\mathcal{N}) \Big)$$
where  $ PD([M_{I_q}(r)]) \in H^{2q-2}(M)$ is the class of $ M_{I_q}(r) \subset
M,$ and  $c_1(\mathcal{N}) $ is the first Chern class of the normal bundle
$\mathcal N$ to $ M_{I_q}(r) \subseteq M.$ 
Moreover $PD([M_{I_q}(r)]) \in H^{2q-2}(M)$ and $c_1(\mathcal{N}) $ are clearly
combinations of the $c_i$'s, since Proposition \ref{MIP} follows from diagram
\eqref{beauty} (which holds for any smooth polygon space, and in particular for
any $r$ such that $M(r) \simeq M$).
Thus, by Theorem \ref{123}, the following holds:

\begin{theorem}\label{fine}
For $r$ generic, the cohomology ring $H^*(M(r), \C)$ of the moduli space of
polygons $M(r)$ is generated by the first Chern classes $c_1,
\ldots, c_n$ of the $n$ complex line bundles associated to the fibration
$\mu^{-1}(r_1, \ldots, r_n) \rightarrow M(r).$ So $$ H^*(M(r), \C) \simeq
\C[x_1, \ldots, x_n]/ \textrm{Ann}(\vol \, M(r))$$ where a polynomial $Q(x_1,
\ldots, x_n) \in \ann(\vol \, M(r)) $  if and only if $$Q \Big( \frac{\d}{\d
r_{1}}, \ldots, \frac{\d}{\d r_{n}} \Big) \vol \, M(r)=0$$
and, as in Theorem \ref{123}, the isomorphism is given by $x_i \mapsto c_i$.   
\end{theorem} 
Note that the formula \eqref{mah} determines not only the cohomology ring 
of the polygon space $M(r)$ but also its intersection numbers. Explicit formulas
for these have been obtained by Agapito and Godinho \cite{ag} via a recursion
relation in $n$, by Takakura \cite{ta01} using ``quantization commutes with
reduction'' and by Konno \cite{konno} using algebro-geometric methods. 
For example, consider the lengths vector $ r= (4,3,4,3,4)$ as in \cite [Example
7.1]{ag} (or equivalently its projection onto $\Xi$). By Theorem \ref{volume}
one calculates that the volume of $M(r)$ is 
$$
\vol \, M(r) = -\pi^2 \Big( 6 \sum_{i=1}^n r_i^2 -2 \sum_{i \neq j} r_ir_j
\Big).
$$
By the formulas \eqref{mah} one recovers (up to rescaling by $2 \pi^2$) the
results in \cite{ag}, precisely
$$\int_{M(r)} c_i^{2} = - 6 \pi^2 \quad \forall \, i=1, \ldots, 5; $$
$$ \int_{M(r)} c_i c_j = 2 \pi^2 \quad \forall \, i \neq j.$$

\begin{remark}\label{hausknut}
In \cite{hk} Hausmann and Knutson compute the cohomology ring
$H^*(M(r), \mathbb Z)$ in terms of generators they call $R$ and $V_i$.
Denote by $\mathcal L$ the collection of $r$-long sets and define the
collection of indeces
$\mathcal L_n$ and $\mathcal S_n$ as follows:
$$
\mathcal L_n:= \big\{J\subset \{1, \ldots, n-1\} \mid J \cup \{n\} \textnormal{
is long} \big\}
$$
$$
\mathcal S_n:= \big\{J\subset \{1, \ldots, n-1\} \mid J \cup \{n\} \textnormal{
is short} \big\}.
$$ 
\begin{theorem} (Hausmann--Knutson)
For $r$ generic, the cohomology ring $H^*(M(r), \mathbb Z)$ is
$$\mathbb Z[R, V_1, \ldots, V_{n-1}] / I_{Pol} $$
where $R$ and $V_i$ are of degree $2$ and $I_{Pol}$ is generated by
the following three families:
\begin{itemize}
 \item $V_i^2 + R_iV_i$ for all $i =1, \ldots n-1$;
\item $\prod_{i \in J} V_i$ for all $J \in \mathcal L_n$;
\item $\sum_{S \subset L, S \in S_n} \big( \prod_{i \in S} V_i\big)R^{| L
\setminus S |-1}$ for all $L\subset \{1, \ldots, n-1\}$ long.
\end{itemize}

\end{theorem}

They also relate the generators $R$ and $V_i$ to the first Chern classes
$\tilde{c}_i:= c_1 (A_i)$ of circle bundles $A_i \to M(r)$ where
$$
A_i:= \{ (e_1, \ldots, e_n) \in \prod_{i=1}^n S^2_{r_i} \mid \sum_{i=1}^n e_i =0
\, \textnormal{and} \, e_i=(0,0,r_i)\}.
$$
Precisely, 
\begin{equation*}
\tilde{c}_i= \left\{\begin{array}{ll}
 R + 2V_i & \textnormal{if } i=1, \ldots, n-1\\
-R& \textnormal{if } i =n.
\end{array} \right.
\end{equation*}
Let $\omega$ be the reduced symplectic form on the polygon space $M(r)$. Then 
$$
\tilde{c}_i= \frac{\partial}{\partial r_i} [\omega].
$$
Hence, by \eqref{-}, the Chern classes $c_i$ in Theorem \ref{fine} are opposite
to the classes
$\tilde{c}_i$, i.e. $c_i = - \tilde{c}_i$. In particular this implies that the
classes
$c_1, \ldots, c_n$ also generate the cohomology of the polygon space
$M(r)$ with coefficients in $\mathbb Z[\frac{1}{2}]$, cf. \cite[Corollary 7.4,
Proposition 7.6]{hk}.
Hausmann and Knutson determine the following relations on the generators $c_i$
\begin{equation*}\label{hk123}
\begin{array}{l}
 \textnormal{(1)} \, c_i^2 = c_n^2  \textnormal{ for all } i=1, \ldots, n;\\
 \textnormal{(2)} \, \prod _{i \in L} (c_i + c_n) \textnormal{ for all } L \in
\mathcal L_n;\\
 \textnormal{(3)} \, c_n^{-1} \Big( \prod _{i \in L} (c_i - c_n) - \prod _{i \in L}
(c_i
+ c_n)\Big) \textnormal{ for all }L \subseteq \{ 1, \ldots,n-1\}, L \textnormal{
long}.\\
\end{array}
\end{equation*}
The relations (1) may also be easily obtained from Theorems \ref{fine} and
\ref{volume} since
$$
\frac{\d}{\d r_i} \epsilon_I(r)^{n-3} = \lambda_I^i (n-3) \epsilon_I(r)^{n-2}
$$
where 
\begin{equation}\label{eq:lambda_i}
\lambda_I^i = \left\{
\begin{array}{ll}
 1 & \text{if } i \in I\\
 -1 & \text{if } i \in I^c\\
\end{array} \right.
 \end{equation}
\end{remark}

\subsubsection{Some examples}
Let $\Delta^{\! 0}$ and $\Delta^{\! 1}$ be the chambers as in
Example \ref{esvolume}. If $r \in \Delta^{\! 0}$, then  $M(r) \simeq \C\P^2$
(see Proposition \ref{outer walls}) and its symplectic volume is $$ \vol \, M(r) =
2 \pi^2 (r_1 +r_2 -r_3 +r_4 +r_5)^2.$$ 
Since
$$ \frac{\d}{\d r_3} \vol \, M(r) = - \frac{\d}{\d r_i} \vol \, M(r) \quad
\forall i = 1,2,4,5$$
it follows that 
$$c_1=c_2=c_4=c_5=-c_3.$$
By Theorem \ref{fine}, the relation on the generator $c_3$ is given by
$\frac{\d^2}{\d r_3^2} \vol \, M(r)= 4 \pi^2 $, hence
$$ H^*(M(r))= \frac{\C[c_3]}{ (c_3^3)}.$$

Now consider $r \in \Delta^{\! 1},$ the polygon space $M(r)$ has symplectic
volume
$$\vol \, M(r)= 4 \pi^2 r_1 ( r_2 - r_3 + r_4 +r_5).$$
It follows that
$$ \frac{\d}{\d r_3} \vol \, M(r) = - \frac{\d}{\d r_i} \vol \, M(r) \quad
\forall i =2,4,5$$
and hence
$$c_2=c_4=c_5=-c_3.$$
The relations on the Chern classes $c_1$ and $c_3$ are given by
$$\frac{\d^2}{\d r_1^2} \vol \, M(r) = 0 \quad \text{and} \quad \frac{\d^2}{\d
r_3^2} \vol \, M(r) = 0.$$
Hence the cohomology ring of $M(r)$ is 
$$  H^*(M(r))= \frac{\C[c_1,c_3]}{(c_1^2, c_3^2)}.$$
By the wall-crossing study, cf. Section \ref{es wall-crossing}, $M(r)$ is
diffeomorphic to $\C\P^2$ blown up at a
point with exceptional divisor $M_{\{1,3\}}$. By Proposition \ref{MIP}
$$PD ([M_{\{1,3\}} (r)])= - (c_1 + c_3).$$
With respect to the basis $\{-c_1 +c_3, - (c_1 + c_3) \}$ the polygon space
$M(r)$ has intersection form 
$\left( \begin{array}{cc}
  -1 &0 \\
0 &1
 \end{array} \right).
 $

Consider now the lengths vector $ r= \frac{1}{11}(3,1,3,1,3)$. Hausmann and
Knutson \cite{hk1} have shown that $M(r)$ is isomorphic to $S^2 \times S^2$.
It is again a plain computation to obtain the volume of $M(r)$:
$$\vol \, M(r)= 8 \pi^2 r_2 r_4.$$
From this we obtain the relations 
$$ c_1= c_3 = c_5 =0$$
and by Theorem \ref{fine} the cohomology of $M(r)$ is 
$$  H^*(M(r))= \frac{\C[c_2,c_4]}{(c_2^2, c_4^2)}.$$
Moreover, with respect to the basis $\{ c_2, c_4\}$ the polygon space $M(r)$ 
has intersection form 
$\left( \begin{array}{cc}
 1 &0 \\
0 &1
 \end{array} \right).
 $
This also illustrates that $H^* (\C \P^2 \sharp \overline{\C\P^2}, \C)\approx
H^* (S^2 \times S^2, \C)$ as indeed the intersection forms 
$\left( \begin{array}{cc}
  -1 &0 \\
0 &1
 \end{array} \right)
 $
and $\left( \begin{array}{cc}
 1 &0 \\
0 &1
 \end{array} \right)
 $ 
are equivalent over $\C$ (in fact, they are equivalent over $\mathbb
Z[\frac{1}{2}]$).

The latter example is a particular case of lengths vectors of type $$r=
\frac{1}{p} (r_1, \ldots, r_{n-3}, 1,1,1)$$ with $\sum_{i=1}^{n-3} r_i < 1$ and
$p=3+ \sum_{i=1}^{n-3} r_i$. In this case the long sets $I$ are all and just the
sets that contain at least two elements of $\{n-2, n-1, n \}$. The volume of the
associated polygon space $M(r)$ is 
\begin{equation}\label{eq:111}
 \vol \, M(r) = \frac{(2 \pi)^{n-3}}{(n-3)!} 2^{n-2} \, r_1 \cdots r_{n-3}.
\end{equation}
Thus $c_{n-2}=c_{n-1}=c_n=0$ and
$$  H^*(M(r))= \frac{\C[c_1,\ldots, c_{n-3}]}{(c_1^2,\ldots, c_{n-3}^2)}.$$
To obtain \eqref{eq:111} from Theorem \ref{volume} one can first observe that
the volume of $M(r)$ can be rewritten as follows
\begin{align*}
\vol \, M(r)& = C \sum_{I \text{long}} (-1)^{n-|I|} \sum_{(k_1, \ldots, k_n) \in
K} \binom{n-3}{k_1, \ldots, k_n} (\lambda_I^1 r_1)^{k_1} \cdots (\lambda_I^n
r_n)^{k_n}\\
 & = C \sum_{(k_1, \ldots, k_n) \in K} \binom{n-3}{k_1, \ldots, k_n} r_1^{k_1}
\cdots r_n^{k_n} \sum_{I \text{long}} (-1)^{n-|I|} (\lambda_I^1)^{k_1} \cdots
(\lambda_I^n)^{k_n}\\
\end{align*}
where $C= - \frac{(2 \pi)^{n-3}}{2(n-3)!}$, $K = \{( k_1, \ldots, k_n) \in
\mathbb Z^n_+ \mid \sum_{i=1}^n k_i = n-3 \}$ and $\lambda_I^i$ is as in
\eqref{eq:lambda_i}. Let us concentrate on the second sum.

For $(k_1, \ldots, k_n) = (1, \ldots, 1, 0,0,0)$ one obtains
\begin{align}\label{eq:000}
\nonumber &\sum_{I \text{long}} (-1)^{n-|I|} \lambda_I^1 \cdots \lambda_I^n= \\
& 3 \sum_{j=0}^{n-3} \binom{n-3}{j} (-1)^{n-2-j} (-1)^{n-3+j}  +
\sum_{j=0}^{n-3} \binom{n-3}{j} (-1)^{n-3-j} (-1)^{n-3+j}= \\
\nonumber & 2 \sum_{j=0}^{n-3} \binom{n-3}{j} = -2^{n-3}
\end{align}
where the first sum in \eqref{eq:000} is relative to long sets $I$ such that 
$| I \cap \{ n-2, n-1, n \}|=2$ and the second sum to long sets $I$ such that 
$\{ n-2, n-1, n \}\subseteq I$.
By similar arguments one can prove that for any other choice of $( k_1, \ldots,
k_n) \in K$ one obtains
$$
\sum_{I \text{long}} (-1)^{n-|I|} (\lambda_I^1)^{k_1} \cdots
(\lambda_I^n)^{k_n} = 0
$$
hence proving \eqref{eq:111}.



\bibliographystyle{amsplain}

\end{document}